\documentclass[numbers,webpdf,imanum]{ima-authoring-template}%
\usepackage{stmaryrd}
\usepackage{float}
\usepackage{microtype}
\usepackage{upgreek}
\usepackage{mathtools}
\usepackage{paper_diening}
\usepackage{mathrsfs}
\usepackage{hyperref}
\usepackage{booktabs}
\usepackage{colortbl}
\usepackage{tabulary}
\usepackage{diagbox}
\usepackage{caption}
\usepackage{esint}

\usepackage{calrsfs}
\DeclareMathAlphabet{\pazocal}{OMS}{zplm}{m}{n}
\DeclareMathAlphabet{\pazobfcal}{OMS}{zplm}{b}{n}

\graphicspath{{Fig/}}

\numberwithin{equation}{section}
\theoremstyle{thmstyletwo}%
\newtheorem{theorem}[equation]{Theorem}
\newtheorem{proposition}[equation]{Proposition}%
\newtheorem{lemma}[equation]{Lemma}%
\newtheorem{remark}[equation]{Remark}%
\newtheorem{corollary}[equation]{Corollary}%

\newtheorem{assumption}[equation]{Assumption}

\providecommand{\meantmp}[2]{#1\langle{#2}#1\rangle}
\providecommand{\mean}[1]{\meantmp{}{#1}}

\providecommand{\Vo}{{\mathaccent23 V}}
\providecommand{\Qo}{{\mathaccent23 Q}}

\providecommand{\SSS}{\mathbf{S}}

\providecommand{\divo}{\mathrm{div}\,}

\newcommand\blfootnote[1]{%
	\begingroup
	\renewcommand\thefootnote{}\footnote{#1}%
	\addtocounter{footnote}{-1}%
	\endgroup
}

\begin{document}



\title[Conditional quasi-optimal error estimate for a FE discretization of the
$p$-NSE]{Conditional quasi-optimal error estimate for a finite element discretization of the
	$p$-Navier--Stokes equations: The case $p>2$}

\author{Alex Kaltenbach*\ORCID{0000-0001-6478-7963}
\address{\orgdiv{Institute of Mathematics}, \orgname{Technical University of Berlin},\\ \orgaddress{\street{Straße des 17.~Juni 135}, \postcode{10623}, \state{Berlin}, \country{Germany}}}}
\corresp[*]{Alex Kaltenbach: \href{email:kaltenbach@math.tu-berlin.de}{kaltenbach@math.tu-berlin.de}}
\author{Michael R\r{u}\v{z}i\v{c}ka
\address{\orgdiv{Department of Applied Mathematics}, \orgname{University of Freiburg},\\  \orgaddress{\street{Ernst--Zermelo-Straße 1}, \postcode{79104}, \state{Freiburg}, \country{Germany}}}}

\authormark{Alex Kaltenbach and Michael R\r{u}\v{z}i\v{c}ka}



\abstract{\hspace*{5mm}In this paper, we derive quasi-optimal \textit{a priori} error estimates for the kinematic pressure~for~a~Finite Element (FE) approximation  of steady systems of $p$-Navier--Stokes type in the case of shear-thickening, \textit{i.e.}, in the case $p>2$, imposing a new mild Muckenhoupt regularity condition.
}
 
\maketitle

\section{Introduction}\blfootnote{* funded by the Deutsche Forschungsgemeinschaft (DFG, German Research Foundation) - 525389262.}\enlargethispage{10mm}

\hspace*{5mm}In the present paper, we examine a \textit{Finite Element (FE)} approximation of 
steady systems of \textit{$p$-Navier--Stokes  type}, \textit{i.e.}, 
\begin{equation}
	\label{eq:p-navier-stokes}
	\begin{aligned}
		-\divo\SSS(\bfD\bfv)+[\nabla\bfv]\bfv+\nabla q&=\bff   \qquad&&\text{in }\Omega\,,\\
		\divo\bfv&=0 \qquad&&\text{in }\Omega\,,
		\\
		\bfv &= \mathbf{0} &&\text{on } \partial\Omega\,\textcolor{black}{,}
	\end{aligned}
\end{equation}
for quasi-optimal \textit{a priori} error estimates for the kinematic pressure  in the case of shear-thickening~fluids, \textit{i.e.},  $p>2$.
The system \eqref{eq:p-navier-stokes} describes the steady motion of a homogeneous,
incompressible fluid with shear-dependent viscosity. More precisely,
for a given vector field $\bff \colon\Omega\to \setR^d$ describing~external~forces, an incompressibility constraint  \eqref{eq:p-navier-stokes}$_2$,  and a non-slip
boundary condition \eqref{eq:p-navier-stokes}$_3$, the system \eqref{eq:p-navier-stokes} seeks for a
\textit{velocity vector field} $\bfv=(v_1,\ldots,v_d)^\top\colon \Omega\to
\setR^d $ and a \textit{kinematic pressure} $q\colon \Omega\to \setR$ solving~\eqref{eq:p-navier-stokes}.
Here, $\Omega\subseteq \mathbb{R}^d$, $d\in \set{2,3}$, is a bounded, polygonal (if $d=2$)~or~polyhedral~(if~${d=3}$)~Lipschitz domain. The \textit{extra stress tensor} $\SSS(\bfD\bfv)\colon\Omega\to \smash{\setR^{d\times d}_{\textup{sym}}}$ depends on the \textit{strain rate tensor} $\bfD\bfv\coloneqq \frac{1}{2}(\nabla\bfv+\nabla\bfv^\top)\colon\Omega\to \smash{\setR^{d\times d}_{\textup{sym}}}$, \textit{i.e.}, the symmetric part of the velocity gradient  $\nabla\bfv\coloneqq (\partial_j v_i)_{i,j=1,\ldots,d}
\colon\Omega\to \setR^{d\times d}$. The \textit{convective term} $\smash{[\nabla\bfv]\bfv\colon\Omega\to \mathbb{R}^d}$~is~defined~via $\smash{([\nabla\bfv]\bfv)_i\coloneqq \sum_{j=1}^d{v_j\partial_j v_i}}$ for all $i=1,\ldots,d$.

Throughout the paper, we  assume that the extra stress tensor~$\SSS$~has~\textit{$(p,\delta)$-structure} (\textit{cf.}~Assumption~\ref{assum:extra_stress}). The relevant example falling into this class is  $\SSS\colon \mathbb{R}^{d\times d}\to \mathbb{R}^{d\times d}_{\textup{sym}}$, for every $\bfA\in \mathbb{R}^{d\times d}$ defined via 
	\begin{align}
		\SSS(\bfA)\coloneqq \mu_0\, (\delta+\vert \bfA\vert)^{p-2}\bfA\,,\label{intro:example}
\end{align}
where $p\in (1,\infty)$, $\delta\ge 0$, and $\mu_0>0$.

The \emph{a priori} error analysis of the steady $p$-Navier--Stokes problem \eqref{eq:p-navier-stokes} using FE approximations is by now well-understood: recently, in \cite{JK23_inhom}, \emph{a priori} error estimates 
in the
 case of shear-thickening, \textit{i.e.}, in the case $p > 2$, were derived, which are optimal for the velocity vector field, but sub-optimal for the kinematic pressure.
 More precisely, this lacuna is mainly due to the following technical hurdle, here, exemplified by the error analysis of the FE approximations of the steady $p$-Navier--Stokes problem \eqref{eq:p-navier-stokes} used in \cite{JK23_inhom}:
in it, it turns out that the error of the kinematic pressure measured in the squared $L^{p'}(\Omega)$-norm is bounded by the squared $L^{(\varphi_{\smash{\vert \bfD\bfv\vert}})^*}(\Omega)$-norm of the extra-stress errors, \textit{i.e.}, it holds that
\begin{align}\label{intro:relation0}
	\begin{aligned}
		\| q_h-q\|_{p',\Omega}^2&\leq \|\bfS(\bfD\bfv_h)-\bfS(\bfD\bfv)\|_{(\varphi_{\smash{\vert \bfD\bfv\vert}})^*,\Omega}^2+\textup{(h.o.t.)}\,.
	\end{aligned}
\end{align}
The relation \eqref{intro:relation0} is a consequence of the squared discrete inf-sup stability result (\textit{cf.}\ \cite[Lem.\ 6.10]{DiPE12})
\begin{align}\label{intro:inf-sup}
	c\,\| z_h\|_{p',\Omega}^2\leq \sup_{\bfz_h\in  \Vo_h\;:\;\|\nabla\bfz_h\|_{p,\Omega}\leq 1}{\big((z_h,\divo \bfz_h)_{\Omega}\big)^2}\,,
\end{align} 
valid for each $z_h\in \Qo_h$, where $\Vo_h\subseteq (W^{1,p}_0(\Omega))^d$ denotes a discrete velocity vector space and 
 and $\Qo_h\subseteq L^{p'}_0(\Omega)$ a discrete pressure space jointly forming a discretely inf-sup-stable finite element couple.
Then, using the estimates (\textit{cf}. \cite[Lem.\ 4.4(ii)]{kr-pnse-ldg-3})
\begin{align*}
	\|\bfS(\bfD\bfv_h)-\bfS(\bfD\bfv)\|_{(\varphi_{\smash{\vert \bfD\bfv\vert}})^*,\Omega}^2\leq 
	c\,\|\bfF(\bfD\bfv_h)-\bfF(\bfD\bfv)\|_{2,\Omega}^2\,,
\end{align*}
from the relation \eqref{intro:relation0}, it follows that
\begin{align}\label{intro:relation0.1}
	\| q_h-q\|_{p',\Omega}^2\leq c\,\|\bfF(\bfD\bfv_h)-\bfF(\bfD\bfv)\|_{2,\Omega}^2+\textup{(h.o.t.)}\,,
\end{align}
\textit{i.e.}, the kinematic pressure error measured in  the squared  $L^{p'}(\Omega)$-norm is bounded by the~squared~velocity vector field error.
However, numerical experiments (\textit{cf.}\ \cite{kr-pnse-ldg-3}) indicate that the relation \eqref{intro:relation0.1} is potentially sub-optimal and suggest instead
the relation 
\begin{align}
	\label{intro:relation}
	\rho_{(\varphi_{\vert \bfD\bfv\vert})^*,\Omega}(q_h-q)\leq c\,\|\bfF(\bfD\bfv_h)-\bfF(\bfD\bfv)\|_{2,\Omega}^2+\textup{(h.o.t.)}\,.
\end{align}
	In \cite{KR24pressure}, in the case of a Local Discontinuous Galerkin (LDG) approximation, such a  discrete analogue has been established under 
	the additional assumption that the viscosity of the fluid is a Muckenhoupt weight of class $2$, \textit{i.e.}, if we have that
\begin{align}\label{intro:Muckenhoupt_regularity}
	\mu_{\bfD\bfv} \coloneqq (\delta+\vert \overline{\bfD\bfv}\vert)^{p-2}\in A_2(\mathbb{R}^d)\,,
\end{align}
where $ \overline{\bfD\bfv}\coloneqq  \bfD\bfv$ a.e.\ in $\Omega$ and $\overline{\bfD\bfv}\coloneqq \bfzero$ a.e.\  in $\mathbb{R}^d\setminus\Omega$. In \cite{kr-nekorn,kr-nekorn-add}, it turned out that the Muckenhoupt regularity 
assumption \eqref{intro:Muckenhoupt_regularity} can not be expected, in general, in the three-dimensional~case. However, in the two-dimensional case, regularity results (\textit{cf.}\ \cite{KMS2}) suggest that \eqref{intro:Muckenhoupt_regularity} is satisfied~under~mild assumptions.
Finally, if the Muckenhoupt regularity 
	condition \eqref{intro:Muckenhoupt_regularity} is satisfied, from~the~relation~\eqref{intro:relation}, we are able to derive an alternative \textit{a priori} error estimate for the pressure which turns out to be quasi-optimal and, thus, is the first quasi-optimal  \textit{a priori} error estimate for the pressure in the~case~of~shear-thickening,~\textit{i.e.},~$p>2$.

\textit{This paper is organized as follows:} 
In Section
\ref{sec:preliminaries}, we introduce the employed~\mbox{notation},~define~relevant function spaces, 
basic assumptions on the extra stress~tensor~$\SSS$~and~its consequences, the weak formulations Problem (\hyperlink{Q}{Q}) and
Problem~(\hyperlink{P}{P})~of~the~system~\eqref{eq:p-navier-stokes}, and the  discrete
operators.~In~Section~\ref{sec:fe}, we introduce the discrete weak
formulations Problem (\hyperlink{Qhldg}{Q$_h$}) and Problem
(\hyperlink{Phldg}{P$_h$}), recall known \textit{a priori} error estimates and
prove the main result of the paper, \textit{i.e.}, a quasi-optimal (with
respect to the Muckenhoupt regularity condition
\eqref{intro:Muckenhoupt_regularity}) \textit{a priori} error estimate for the
kinematic pressure in the case $p>2$ (\textit{cf.} \hspace*{-0.1mm}Theorem~\hspace*{-0.1mm}\ref{thm:error_pressure_optimal}, \hspace*{-0.1mm}Corollary \hspace*{-0.1mm}\ref{cor:error_pressure_optimal}).
\hspace*{-0.1mm}In \hspace*{-0.1mm}Section \hspace*{-0.1mm}\ref{sec:experiments}, \hspace*{-0.1mm}we \hspace*{-0.1mm}review~\hspace*{-0.1mm}the~\hspace*{-0.1mm}theoretical~\hspace*{-0.1mm}findings~\hspace*{-0.1mm}via~\hspace*{-0.1mm}numerical~\hspace*{-0.1mm}experiments.

\newpage
\section{Preliminaries}\label{sec:preliminaries}

\subsection{Basic Notation}

\hspace*{5mm}We employ $c,C>0$ to denote generic constants,~that~may change from line to line, but do not depend on the crucial quantities. Moreover, we~write~$u\sim v$ if and only if there exist constants $c,C>0$ such that $c\, u \leq v\leq C\, u$.

For $k\in \setN$ and $p\in [1,\infty]$, we employ the customary
Lebesgue spaces $(L^p(\Omega), \|\cdot\|_{p,\Omega}) $ and Sobolev
spaces $(W^{k,p}(\Omega),\|\cdot\|_{k,p,\Omega})$, where $\Omega
\subseteq \setR^d$, $d\in \{2,3\}$, is a bounded, polygonal (if $d=2$) or polyhedral (if $d=3$) Lipschitz domain. The space $\smash{W^{1,p}_0(\Omega)}$
is defined as those functions from $W^{1,p}(\Omega)$ whose traces vanish on $\partial\Omega$. 
The Hölder dual exponent is defined via $p' = \tfrac{p}{p-1} \in [1,
\infty]$.  
We denote vector-valued functions by boldface letters~and~tensor-valued
functions by capital boldface letters. 
The Euclidean inner product
between two vectors $\bfa =(a_1,\ldots,a_d)^\top,\bfb=(b_1,\ldots,b_d)^\top\in \mathbb{R}^d$ is denoted by 
$\bfa \cdot\bfb\coloneqq \sum_{i=1}^d{a_ib_i}$, while the
Frobenius inner product between two tensors $\bfA=(A_{ij})_{i,j\in \{1,\ldots,d\}},$ $\bfB=(B_{ij})_{i,j\in \{1,\ldots,d\}}\in \mathbb{R}^{d\times d }$ is denoted by
$\bfA: \bfB\coloneqq \sum_{i,j=1}^d{A_{ij}B_{ij}}$.  
Moreover, for a (Lebesgue) measurable set $M\subseteq \mathbb{R}^n$, $n\in \mathbb{N}$,  and (Lebesgue) measurable functions, vector or tensor fields $\mathbf{u},\mathbf{w}\in (L^0(M))^{\ell}$, $\ell\in  \mathbb{N}$, we write\vspace*{-1.5mm}
\begin{align*}
(\mathbf{u},\mathbf{w})_M \coloneqq \int_M \mathbf{u} \odot  \mathbf{w}\,\mathrm{d}x	\,,\\[-7mm]
\end{align*}
whenever the right-hand side is well-defined, where $\odot\colon \mathbb{R}^{\ell}\times \mathbb{R}^{\ell}\to \mathbb{R}$ either denotes scalar multiplication, the Euclidean inner  product or the Frobenius inner product.
The integral mean  of an integrable function, vector or tensor field $\mathbf{u}\in (L^0(M))^{\ell}$, $\ell\in  \mathbb{N}$,  over a (Lebesgue) measurable set $M\subseteq \mathbb{R}^n$, $n\in \mathbb{N}$,  with $\vert M\vert >0$ is denoted~by\enlargethispage{10mm}
$
	\mean{\mathbf{u}}_M \coloneqq \smash{\frac 1 {|M|}\int_M \mathbf{u} \, \mathrm{d}x}
$.

\subsection{N-functions and Orlicz spaces}

\hspace*{5mm}A convex function $\psi \colon \setR^{\geq 0} \to \setR^{\geq 0}$ is called \textit{N-function} if it holds that~${\psi(0)=0}$,~${\psi(t)>0}$~for~all~${t>0}$, $\lim_{t\rightarrow0} \psi(t)/t=0$, and
$\lim_{t\rightarrow\infty} \psi(t)/t=\infty$. 
A Carath\'eodory function $\psi \colon M \times \setR^{\geq 0} \to \setR^{\geq 0}$, where $M\subseteq \mathbb{R}^n$, $n\in \mathbb{N}$, is a (Lebesgue) measurable set, such that $\psi(x,\cdot)$ is an N-function for a.e.\ $x \in M$, is called \textit{generalized N-function}. \textcolor{black}{The
	modular~is~defined~via
	$\rho_\psi(f)\coloneqq \rho_{\psi,\Omega}(f)\coloneqq \int_\Omega
	\psi(\abs{f})\,\textup{d}x $ if $\psi$ is an N-function, and via
	$\rho_\psi(f)\coloneqq \rho_{\psi,\Omega}(f)\coloneqq \int_\Omega
	\psi(x,\abs{f(x)})\,\textup{d}x $, if $\psi$ is a generalized
	N-function.}  We~define the \textit{(convex) conjugate N-function}
We~define the \textit{(convex) conjugate N-function} $\psi^*\colon M\times \setR^{\geq 0} \to \setR^{\geq 0}$ via
${\psi^*(t,x)\coloneqq \sup_{s \geq 0} (st - \psi(s,x))}$ for all $t \ge 0$ and a.e.\ $x\in M$, which satisfies
$(\partial_t(\psi^*))(x,t) =  (\partial_t\psi)^{-1}(x,t)$ for all $t\ge 0$ and a.e.\ $x\in M$. A (generalized) N-function
$\psi$ satisfies the \textit{$\Delta_2$-condition} (in short,
$\psi \in \Delta_2$), if there exists
$K> 2$ such that ${\psi(2\,t,x) \leq K\,\psi(t,x)}$ for all
$t \geq 0$ and a.e.\ $x\in M$. The smallest such constant is denoted by
$\Delta_2(\psi) > 0$. 
We need the following version of the \textit{$\epsilon$-Young inequality}: for every
${\epsilon> 0}$, there exists a constant $c_\epsilon>0 $,
depending only on $\Delta_2(\psi),\Delta_2( \psi ^*)<\infty$, such
that for every $s,t\geq 0$ and a.e.\ $x\in M$, it holds that
\begin{align} \label{ineq:young}
	s\,t&\leq c_\epsilon \,\psi^*(s,x)+ \epsilon \, \psi(t,x)\,.
\end{align}

\subsection{Basic properties of the extra stress tensor}

\hspace*{5mm}Throughout~the~entire~paper, we will always assume that the extra stress tensor 
$\SSS$
has $(p,\delta)$-structure. A detailed
discussion and full proofs can be found, \textit{e.g.}, in
\cite{die-ett,dr-nafsa}. For a given tensor $\bfA\in \setR^{d\times d}$, we denote its symmetric part by
${\bfA^{\textup{sym}}\coloneqq \frac{1}{2}(\bfA+\bfA^\top)\in
	\setR^{d\times d}_{\textup{sym}}\coloneqq \{\bfA\in \setR^{d\times
		d}\mid \bfA=\bfA^\top\}}$.

For $p \in (1,\infty)$~and~$\delta\ge 0$, we define the \textit{special N-function}
$\phi=\phi_{p,\delta}\colon\setR^{\ge 0}\to \setR^{\ge 0}$ via\vspace*{-1.5mm}
\begin{align} 
	\label{eq:def_phi} 
	\varphi(t)\coloneqq  \int _0^t \varphi'(s)\, \mathrm ds,\quad\text{where}\quad
	\varphi'(t) \coloneqq  (\delta +t)^{p-2} t\,,\quad\textup{ for all }t\ge 0\,.
\end{align}
An important tool in our analysis \textcolor{black}{plays} {\rm shifted N-functions}
$\{\psi_a\}_{\smash{a \ge 0}}$ (\textit{cf}.\ \cite{DK08,dr-nafsa}). For a given N-function $\psi\colon\mathbb{R}^{\ge 0}\to \mathbb{R}^{\ge
	0}$, we define the family  of \textit{shifted N-functions} ${\psi_a\colon\mathbb{R}^{\ge
		0}\to \mathbb{R}^{\ge 0}}$,~${a \ge 0}$,~via
\begin{align}
	\label{eq:phi_shifted}
	\psi_a(t)\coloneqq  \int _0^t \psi_a'(s)\, \mathrm ds\,,\quad\text{where }\quad
	\psi'_a(t)\coloneqq \psi'(a+t)\frac {t}{a+t}\,,\quad\textup{ for all }t\ge 0\,.
\end{align}

\begin{assumption}[Extra stress tensor]\label{assum:extra_stress} 
	We assume that the extra stress tensor $\SSS\colon\setR^{d\times d}\to \setR^{d\times d}_{\textup{sym}}$ belongs to $C^0(\setR^{d\times d}; \setR^{d\times d}_{\textup{sym}})\cap C^1(\setR^{d\times d}\setminus\{\mathbf{0}\}; \setR^{d\times d}_{\textup{sym}}) $ and satisfies $\SSS(\bfA)=\SSS(\bfA^{\textup{sym}})$ for all $\bfA\in \setR^{d\times d}$ and $\SSS(\mathbf{0})=\mathbf{0}$. Moreover, we assume~that~the~tensor $\SSS=(S_{ij})_{i,j=1,\ldots,d}$ has \textup{$(p,\delta)$-structure}, \textit{i.e.},
	for some $p \in (1, \infty)$, $ \delta\in [0,\infty)$, and the
	N-function $\varphi=\varphi_{p,\delta}$ (\textit{cf}.~\eqref{eq:def_phi}), there
	exist constants $C_0, C_1 >0$ such that
	\begin{align}
		\sum\limits_{i,j,k,l=1}^d \partial_{kl} S_{ij} (\bfA)
		B_{ij}B_{kl} &\ge C_0 \, \frac{\phi'(|\bfA^{\textup{sym}}|)}{|\bfA^{\textup{sym}}|}\,|\bfB^{\textup{sym}}|^2\,,\label{assum:extra_stress.1}
		\\
		\big |\partial_{kl} S_{ij}({\bfA})\big | &\le C_1 \, \frac{\phi'(|\bfA^{\textup{sym}}|)}{|\bfA^{\textup{sym}}|}\label{assum:extra_stress.2}
	\end{align}
	are satisfied for all $\bfA,\bfB \in \setR^{d\times d}$ with $\bfA^{\textup{sym}}\neq \mathbf{0}$ and all $i,j,k,l=1,\ldots,d$.~The~constants $C_0,C_1>0$ and $p\in (1,\infty)$ are called the \textup{characteristics} of $\SSS$. \enlargethispage{2mm}
\end{assumption}

\begin{remark}
	\begin{itemize}[{(ii)}]
		\item[(i)] Assume that $\SSS$ satisfies Assumption \ref{assum:extra_stress} for some 
		$\delta \in [0,\delta_0]$.~Then,~if~not~otherwise stated, the
		constants in the estimates depend only on the
                characteristics of\ $\SSS$ and $\delta_0\ge 0$, but are independent of $\delta\ge 0$.
		
		\item[(ii)] Let $\phi$ be defined in \eqref{eq:def_phi} and 
		$\{\phi_a\}_{a\ge 0}$ be the corresponding family of shifted \mbox{N-functions}. Then, the operators 
		$\SSS_a\colon\mathbb{R}^{d\times d}\to \smash{\mathbb{R}_{\textup{sym}}^{d\times
				d}}$, $a \ge 0$, for every $a \ge 0$
		and~$\bfA \in \mathbb{R}^{d\times d}$ defined via 
		\begin{align}
			\label{eq:flux}
			\SSS_a(\bfA) \coloneqq 
			\frac{\phi_a'(\abs{\bfA^{\textup{sym}}})}{\abs{\bfA^{\textup{sym}}}}\,
			\bfA^{\textup{sym}}\,, 
		\end{align}
		have $(p, \delta +a)$-structure.  In this case, the characteristics of
		$\SSS_a$ depend~only~on~${p\in (1,\infty)}$ and are independent of
		$\delta \geq 0$ and $a\ge 0$.
	\end{itemize}
\end{remark}

Closely related to the extra stress tensor $\SSS\colon \setR^{d\times d}\to \setR^{d\times d}_{\textup{sym}}$ with
$(p,\delta)$-structure 
is the non-linear mapping
$\bfF\colon\setR^{d\times d}\to \setR^{d\times d}_{\textup{sym}}$, 
for every $\bfA\in \mathbb{R}^{d\times d}$ defined via
\begin{align}
	\begin{aligned}
		\bfF(\bfA)&\coloneqq (\delta+\vert \bfA^{\textup{sym}}\vert)^{\frac{p-2}{2}}\bfA^{\textup{sym}}\,.
	\end{aligned}\label{eq:def_F}
\end{align}
The connections between
$\SSS,\bfF\colon \setR^{d \times d}
\to \setR^{d\times d}_{\textup{sym}}$ and
$\phi_a,(\phi_a)^*\colon \setR^{\ge
	0}\to \setR^{\ge
	0}$,~${a\ge  0}$, are best explained
by the following proposition (\textit{cf}.~\cite{die-ett,dr-nafsa,dkrt-ldg}).

\begin{proposition}
	\label{lem:hammer}
	Let $\SSS$ satisfy Assumption~\ref{assum:extra_stress}, let $\varphi$ be defined in \eqref{eq:def_phi}, and let $\bfF$ be defined in \eqref{eq:def_F}. Then, uniformly with respect to 
	$\bfA, \bfB \in \setR^{d \times d}$, we have that\vspace{-1mm}
	\begin{align}
		\big(\SSS(\bfA) - \SSS(\bfB)\big):(\bfA-\bfB )
		&\sim  \vert  \bfF(\bfA) - \bfF(\bfB)\vert ^2 \notag
		\\
		&\sim \phi_{\vert \bfA^{\textup{sym}}\vert }(\vert \bfA^{\textup{sym}}
		- \bfB^{\textup{sym}}\vert) \label{eq:hammera}
		\\
		&\sim(\varphi_{\vert\bfA^{\textup{sym}}\vert})^*(\abs{\SSS(\bfA ) - \SSS(\bfB )})\,,\notag
		\\
		\label{eq:hammere}
		\abs{\SSS(\bfA) - \SSS(\bfB)}
		&\sim  \smash{\phi'_{\abs{\bfA}}(\abs{\bfA - \bfB})}\,.
	\end{align}
	The constants 
	depend only on the characteristics of ${\SSS}$.
\end{proposition} 
\begin{remark}\label{rem:sa}
	{\rm
		For the operators $\SSS_a\hspace{-0.1em}\colon\hspace{-0.1em}\mathbb{R}^{d\times d}\hspace{-0.1em}\to\hspace{-0.1em}\smash{\mathbb{R}_{\textup{sym}}^{d\times
				d}}$, $a \ge  0$, defined~in~\eqref{eq:flux},~the~assertions of Proposition~\ref{lem:hammer} hold with $\phi\colon\mathbb{R}^{\ge 0}\to \mathbb{R}^{\ge 0}$ replaced
		by $\phi_a\colon\mathbb{R}^{\ge 0}\to \mathbb{R}^{\ge 0}$, $a\ge 0$.}
\end{remark}

The following results can be found in~\cite{DK08,dr-nafsa}. 

\begin{lemma}[Change of Shift]\label{lem:shift-change}
	Let $\varphi$ be defined in \eqref{eq:def_phi} and let $\bfF$ be defined in \eqref{eq:def_F}. Then,
	for each $\varepsilon>0$, there exists $c_\varepsilon\geq 1$ (depending only
	on~$\varepsilon>0$ and the characteristics of $\phi$) such that for every $\bfA,\bfB\in\smash{\setR^{d \times d}_{\textup{sym}}}$ and $t\geq 0$, it holds that
	\begin{align}
		\smash{\phi_{\vert \bfB\vert}(t)}&\leq \smash{c_\varepsilon\, \phi_{\vert \bfA\vert}(t)
			+\varepsilon\, \vert \bfF(\bfB) - \bfF(\bfA)\vert ^2\,,}\label{lem:shift-change.1}
		\\
		\smash{\phi_{\vert \bfB\vert}(t)}&\leq \smash{c_\varepsilon\, \phi_{\vert \bfA\vert} (t)
			+\varepsilon\, \phi_{\vert \bfA\vert}(\vert \vert \bfB\vert - \vert \bfA\vert\vert )\,,}\label{lem:shift-change.2}
		\\
		\smash{(\phi_{\vert \bfB\vert})^*(t)}&\leq \smash{c_\varepsilon\, (\phi_{\vert \bfA\vert})^*(t)
			+\varepsilon\, \vert \bfF(\bfB) - \bfF(\bfA)\vert^2} \,,\label{lem:shift-change.3}
		\\
		\smash{(\phi_{\vert \bfB\vert})^*(t)}&\leq \smash{c_\varepsilon\, (\phi_{\vert \bfA\vert})^*(t)
			+\varepsilon\, \phi_{\vert \bfA\vert}(\vert \vert \bfB\vert - \vert \bfA\vert\vert )}\,.\label{lem:shift-change.4}
	\end{align}
\end{lemma}

\subsection{The $p$-Navier--Stokes system} 
\hspace*{5mm}Let us briefly recall some well-known facts about the $p$-Navier--Stokes equations
\eqref{eq:p-navier-stokes}. For $p\in (1,\infty)$, we define the function spaces\\[-4.5mm]
\begin{align*}
	\Vo\coloneqq (W^{1,p}_0(\Omega))^d\,,\qquad
	\Qo\coloneqq L_0^{p'}(\Omega)\coloneqq \big\{z\in
	L^{p'}(\Omega)\;|\;\mean{z}_{\Omega}=0\big\}\,.\\[-6mm] 
\end{align*}
With this notation, assuming that $p\ge \frac{3d}{d+2}$, the weak formulation of the $p$-Navier--Stokes equations \eqref{eq:p-navier-stokes} as a non-linear saddle point problem is the following:

\textit{Problem (Q).}\hypertarget{Q}{} For given $\bff \in (L^{p'}(\Omega))^d$, find $(\bfv,q)\in \Vo \times \Qo$ such that  for all $(\bfz,z)^\top\in \Vo\times Q $, it holds\linebreak\hspace*{4.5mm} that
\begin{align}
	(\SSS(\bfD\bfv),\bfD\bfz)_\Omega+([\nabla\bfv]\bfv,\bfz)_\Omega-(q,\divo\bfz)_\Omega&=(\bff ,\bfz)_\Omega\label{eq:q1}\,,\\
	(\divo\bfv,z)_\Omega&=0\label{eq:q2}\,.
\end{align}
Alternatively, we can reformulate Problem (\hyperlink{Q}{Q}) \textit{``hiding''} the kinematic pressure.

\textit{Problem (P).}\hypertarget{P}{} For given $\bff \in (L^{p'}(\Omega))^d$, find $\bfv\in \Vo(0)$ 
such that for all $\bfz\in \Vo(0) $, it holds that
\begin{align}
	(\SSS(\bfD\bfv),\bfD\bfz)_\Omega+([\nabla\bfv]\bfv,\bfz)_\Omega&=(\bff ,\bfz)_\Omega\,,\label{eq:p}
\end{align}
\hspace*{5mm}where $\Vo(0)\coloneqq \{\bfz\in \Vo\mid \divo \bfz=0\}$.\\
The names \textit{Problem (\hyperlink{Q}{Q})} and \textit{Problem (\hyperlink{P}{P})} are traditional in the literature (\textit{cf}.\  \cite{BF1991,bdr-phi-stokes}).~The~well-posedness of Problem (\hyperlink{Q}{Q}) and Problem (\hyperlink{P}{P}) is usually established in two steps:
first, using pseudo-monotone operator theory (\textit{cf}.\ \cite{lions-quel}), the well-posedness of Problem (\hyperlink{P}{P}) is shown; then, given the well-posedness of Problem (\hyperlink{P}{P}), the well-posedness of Problem (\hyperlink{Q}{Q})
follows using DeRham's lemma. There holds the following regularity property of the
pressure if the velocity satisfies a natural regularity assumption.
\begin{lemma}\label{lem:pres}
	Let $\SSS$ satisfy Assumption~\ref{assum:extra_stress} with
	$p\in (2,\infty)$ and $\delta >0$, and let $(\bfv,q)^\top \in \Vo(0)\times \Qo$
	be a weak solution of Problem (\hyperlink{Q}{Q}) with $\bfF(\bfD \bfv) \in
	(W^{1,2}(\Omega))^{d\times d} $. Then, the following statements apply:
	\begin{itemize}[{(ii)}]
		\item[(i)] If $\bff \in  (L^{p'}(\Omega))^d$, then $\nabla q \in  (L^{p'}(\Omega))^d$.
		\item[(ii)] If $\bff \in (L^2(\Omega))^d$, then $(\delta
		+\vert\bfD\bfv\vert)^{2-p}\vert\nabla q\vert ^2 \in  L^1(\Omega)$.
	\end{itemize}
\end{lemma}\vspace*{-5mm}
\begin{proof} See \cite[Lem.\ 2.6]{kr-pnse-ldg-2}.
\end{proof}

\newpage
\subsection{Discussion of Muckenhoupt regularity condition}\label{sec:reg-assumption}

\hspace*{5mm}In this subsection, we examine the \textit{Muckenhoupt regularity condition}
\begin{align}
	\mu_{\bfD\bfv} \coloneqq (\delta+\vert \overline{\bfD\bfv}\vert)^{p-2}\in A_2(\mathbb{R}^d)\,,\label{eq:reg-assumption}
\end{align}
for a solution $\bfv\in (W^{1,p}_0(\Omega))^d$ of Problem (\hyperlink{P}{P}) (or Problem (\hyperlink{Q}{Q}), respectively), where $\overline{\bfD\bfv}\in (L^p(\mathbb{R}^d))^{d\times d}$ is defined via\vspace*{-2mm}\enlargethispage{6mm}
\begin{align*}
	\overline{\bfD\bfv}\coloneqq\begin{cases}
		\bfD\bfv &\text{ a.e.\ in }\Omega\,,\\
		\bfzero  &\text{ a.e.\ in }\mathbb{R}^d\setminus\Omega\,.
	\end{cases}
\end{align*}
In this connection, recall that for given $p\in [1,\infty)$, a weight  $\sigma\colon \mathbb{R}^d\to (0,+\infty)$, \textit{i.e.}, $\sigma \in L^1_{\textup{loc}}(\mathbb{R}^d)$ and $0<\sigma(x)<+\infty$ for a.e.\ $x\in \mathbb{R}^d$, is said to satisfy the \textit{$A_p$-condition}, if
\begin{align*}
	[\sigma]_{A_p(\mathbb{R}^d)}\coloneqq 	\sup_{B\subseteq \mathbb{R}^d\,:\,B\text{ is a ball}}{\langle \sigma\rangle_B(\langle \sigma^{1-p'}\rangle_B)^{p-1}}<\infty\,.
\end{align*}
We denote by $A_p(\mathbb{R}^d)$ the \textit{class of all weights
	satisfying the $A_p$-condition} and use \textit{weighted Lebesgue
	spaces} $L^p(\Omega;\sigma)$ equipped with the norm
$\|\cdot\|_{p,\sigma,\Omega}\coloneqq  (\int_\Omega
\vert \cdot\vert^p\, \sigma \, \mathrm{d}x)^{1/p}$.

In two dimensions, the Muckenhoupt regularity condition \eqref{eq:reg-assumption} is satisfied under mild assumptions.

\begin{theorem}\label{thm:reg_two_dim}
	Let $\Omega\subseteq \mathbb{R}^2$  be a bounded domain with
        $C^2$-boundary, $p>2$, and $\bff \in (L^s(\Omega))^2$,~where ${s>2}$. Then, there exist $q>2$, $\alpha>0$, and a solution $(\bfv,q)^\top\in \Vo(0)\times \Qo$ of Problem (\hyperlink{Q}{Q}) with the following properties:\vspace*{-2mm}
	\begin{itemize}[{(iii)}]
		\item[(i)] $(\bfv,q)^\top\in (W^{2,q}(\Omega))^2\times W^{1,q}(\Omega)$;
		\item[(ii)] $(\bfv,q)^\top\in (C^{1,\alpha}(\overline{\Omega}))^2\times C^{0,\alpha}(\overline{\Omega})$.
	\end{itemize}
\end{theorem}\vspace*{-5mm}

\begin{proof}
	See \cite[Thm. 6.1]{KMS2}.
\end{proof}

\begin{remark}\label{rem:muckenhoupt}
	\begin{itemize}[{(ii)}]
		\item[(i)] For $\vert \bfD\bfv\vert\in
                  L^\infty(\Omega)$, $\delta>0$,  and $p>2$, it follows $\delta^{p-2}\leq\mu_{\bfD\bfv}\leq (\delta+\|\bfD\bfv\|_{\infty,\Omega})^{p-2}$ a.e.\ in $\Omega$, 
		so that the Muckenhoupt regularity condition \eqref{eq:reg-assumption} is trivially satisfied. As a result, under the assumptions of Theorem \ref{thm:reg_two_dim}, the regularity assumption \eqref{eq:reg-assumption} is satisfied.
		\item[(ii)] We believe that it is possible to prove Theorem \ref{thm:reg_two_dim} without the $C^2$-boundary assumption for polygonal, convex domains.
	\end{itemize}
\end{remark}

The following result implies that, in three dimensions, one cannot hope for the regularity assumption to be satisfied, in general.

\begin{theorem}\label{thm:counter_example}
	Let $\Omega\subseteq \mathbb{R}^d$, $d\ge 2$,  be a bounded domain with $C^{1,1}$-boundary. Then, there exists a vector field $\bfv\colon \Omega\to \mathbb{R}^d$ with the following properties:\vspace*{-2mm}
	\begin{itemize}[{(iii)}]
		\item[(i)] $\bfv\in (W^{2,r}(\Omega))^d\cap (W^{1,s}_0(\Omega))^d$ for all $r\in (1,d)$ and $s\in (1,\infty)$;
		\item[(ii)] $\textup{tr}_{\partial\Omega}(\nabla \bfv) \not\equiv \bfzero$;
		\item[(iii)] $\divo \bfv=  0$ a.e.\ in $\Omega$;
		\item[(iv)] $\mu_{\bfD\bfv}\coloneqq(\delta+\vert \overline{\bfD\bfv}\vert)^{p-2}\notin A_2(\mathbb{R}^d)$ for all $p\in (1,\infty)\setminus\{2\}$ and $\delta\in [0,1]$.
	\end{itemize}
\end{theorem}\vspace*{-5mm}

\begin{proof}
	See \cite[Thm. 2.1]{kr-nekorn} and \cite[Thm. 2]{kr-nekorn-add}.
\end{proof}

\begin{remark}
	Let $\bfv \colon \Omega\to \mathbb{R}^3$ be the vector field from Theorem \ref{thm:counter_example} in the case $d = 3$. Setting $\bff \coloneqq \divo(\bfv\otimes\bfv)-\divo \bfS(\bfD\bfv)$, where $\bfS(\bfD\bfv)\coloneqq\mu_{\bfD\bfv}\bfD\bfv$ and $\mu_{\bfD\bfv}=(\delta+\vert \bfD\bfv\vert)^{p-2}\in A_2(\mathbb{R}^3)$, we see that $\bfv$ and $q \equiv 0$ are a weak solution of the $p$-Navier--Stokes system \eqref{eq:p-navier-stokes} for any $p \in  (1, \infty)$ and $\delta\in [0,1]$. Due  to \cite[Rem.\ 2.3]{kr-nekorn} and \cite{kr-nekorn-add}, we have that $\bfF(\bfD\bfv)\in (W^{1,2}(\Omega))^{3\times 3}$ and $\bff \in (L^2(\Omega))^3$~for~all~$p \in (1, \infty)$.
\end{remark}

\section{Finite Element (FE) approximation}\label{sec:fe}

\subsection{Triangulations}

\hspace*{5mm}Throughout the entire paper, we denote by $\{\mathcal{T}_h\}_{h>0}$ a family of triangulations of $\Omega\subseteq \mathbb{R}^d$, $d\in\{2,3\}$, consisting of $d$-dimensional simplices (\textit{cf}.\ \cite{EG21}). 
Here, $h>0$ refers to the \textit{maximal~mesh-size}, \textit{i.e.}, if we set $h_K\coloneqq  \textup{diam}(K)$ for all $K\in \mathcal{T}_h$, then $h 
	\coloneqq \max_{K\in \mathcal{T}_h}{h_K}
$.
For every $K \in \mathcal{T}_h$,
we denote~by~$\rho_K>0$, the supremum of diameters of inscribed balls~contained~in~$K$. We assume that there is a constant~$\kappa_0>0$, independent of $h>0$, such that $\max_{K\in \mathcal{T}_h}{h_K}{\rho_K^{-1}}\le
\kappa_0$. The smallest such constant is called the \textit{chunkiness} of $\{\mathcal{T}_h\}_{h>0}$. For every $K\in \mathcal{T}_h$, the \textit{element patch} is defined via $\omega_K\coloneqq \bigcup\{K'\in\mathcal{T}_h\mid K'\cap K\neq \emptyset\}$.

\subsection{Finite element spaces and projectors}

\hspace{5mm}Given
$m \in \mathbb N_0$ and $h>0$, we denote by $\mathbb{P}^m(\mathcal{T}_h)$ the space of (possibly discontinuous) scalar functions that are polynomials of degree at most $m$ on each simplex $T\in \mathcal{T}_h$, and set $\mathbb{P}^m_c(\mathcal{T}_h)\coloneqq \mathbb{P}^m(\mathcal{T}_h)\cap C^0(\overline{\Omega})$.
Then, given $k\in\mathbb{N}$ and $\ell \in \mathbb N_0$, we~denote~by
\begin{align}
	\begin{aligned}
		V_h&\subseteq (\mathbb{P}^k_c(\mathcal{T}_h))^d\,, &&\,\Vo_h\coloneqq V_h\cap \Vo\,,\\
		Q_h&\subseteq \mathbb{P}^\ell(\mathcal{T}_h)\,, &&\Qo_h \coloneqq Q_h\cap \Qo\,,
	\end{aligned}
\end{align}
appropriate conforming
finite element spaces such that the following two assumptions are satisfied.

\begin{assumption}
	\label{ass:PiY}
	We assume that $\setR
	\subseteq Q_h$ and that there exists a linear projection operator
	$\Pi_h^Q\colon Q \to Q_h$ which is \textup{locally $L^1$-stable}, \textit{i.e.}, for every $q\in Q$ and $K\in \mathcal{T}_h$, it holds that
	\begin{align}
		\label{eq:PiYstab}
		\langle \vert \Pi_h^Q q\vert \rangle_K &\leq c\, \langle \vert q\vert \rangle_{\omega_K}\,.
	\end{align}
\end{assumption} 

\begin{assumption}[Projection operator $\Pi_h^{\smash{V}}$]\label{ass:proj-div}
	We assume that $\mathbb{P}^1_c(\pazocal{T}_h) \subseteq V_h$ and that there
	exists a linear projection operator $\Pi_h^{\smash{V}}\colon  V \to V_h$ with the following properties:
	\begin{itemize}[{(iii)}]
		\item[(i)] \textup{Preservation of divergence in the $Q_h^*$-sense:} For every $\bfz \in V$ and  $z_h \in Q_h$, it holds that
		\begin{align}
			\label{eq:div_preserving}
			(\divo \bfz,z_h)_\Omega &= (\divo\Pi_h^{\smash{V}}
			\bfz,z_h)_\Omega \,;
		\end{align}
		\item[(ii)] \textup{Preservation of zero boundary values:} It holds $\Pi_h^{\smash{V}}(\Vo) \subseteq \Vo_h$;
		\item[(iii)] \textup{Local $L^1$-$W^{1,1}$-stability:} For every $\bfz \in V$ and $K\in \mathcal{T}_h$, it holds that
		\begin{align}
			\label{eq:Pidivcont}
			\langle \vert \Pi_h^{\smash{V}}\bfz\vert \rangle_K &\leq c\,(\langle \vert \bfz\vert \rangle_{\omega_K} +h_K\, \langle \vert \nabla\bfz\vert \rangle_{\omega_K} )\,.
		\end{align}
	\end{itemize}
\end{assumption}

Next, we present a short list of common mixed finite element spaces $\{V_h\}_{h>0}$ and $\{Q_h\}_{h>0}$ with projections $\{\Pi_h^{\smash{V}}\}_{h>0}$~and~$\{\Pi_h^{\smash{Q}}\}_{h>0}$ on regular triangulations $\{\mathcal{T}_h\}_{h>0}$ satisfying both  Assumption~\ref{ass:PiY} and Assumption~\ref{ass:proj-div}; for a detailed presentation, we recommend the textbook \cite{BBF13}.

\begin{remark}\label{FEM.Q}
	The following discrete spaces and projectors satisfy Assumption~\ref{ass:PiY}:
	\begin{itemize}[{(ii)}]
		\item[(i)] If $Q_h= \mathbb{P}^\ell(\mathcal{T}_h)$ for some $\ell\ge 0$, then $\Pi_h^Q$ can be chosen as (local) $L^2$-projection operator or, more generally, a Cl\'ement type quasi-interpolation operator.
		
		\item[(ii)] If  $Q_h=\mathbb{P}^\ell_c(\mathcal{T}_h)$  for  some  $\ell\ge 1$,  then  $\Pi_h^Q$  can  be  chosen  as  a  Cl\'ement  type  quasi-interpolation operator.
		
	\end{itemize}
\end{remark}

\begin{remark}\label{FEM.V} 
	The following discrete spaces  and projectors satisfy Assumption~\ref{ass:proj-div}:
	\begin{itemize}[{(iii)}]
		\item[(i)] The \textup{MINI element} for $d\in \{2,3\}$, \textit{i.e.}, $X_h=(\mathbb{P}^1_c(\mathcal{T}_h)\bigoplus\mathbb{B}(\mathcal{T}_h))^d$, where $\mathbb{B}(\mathcal{T}_h)$ is the bubble function space, and $Q_h=\mathbb{P}^1_c(\mathcal{T}_h)$, introduced in \cite{ABF84} for $d=2$; see also \cite[Chap.\ II.4.1]{GR86} and \cite[Sec.\ 8.4.2, 8.7.1]{BBF13}. A proof of Assumption \ref{ass:proj-div}
		is given in \cite[Appx.\ A.1]{bdr-phi-stokes}; see also \cite[Lem.\ 4.5]{GL01}~or~\mbox{\cite[p.~990]{DKS13}}.
		
		\item[(ii)] The \textup{Taylor--Hood element} for $d\in\{2,3\}$, \textit{i.e.}, $X_h=(\mathbb{P}^2_c(\mathcal{T}_h))^d$ and $Q_h=\mathbb{P}^1_c(\mathcal{T}_h)$, introduced in \cite{TH73} for $d=2$; see  also \cite[Chap.\ II.4.2]{GR86}, and its generalizations; see, \textit{e.g.}, \cite[Sec.\ 8.8.2]{BBF13}. A proof
		of Assumption \ref{ass:proj-div} is given in \cite[Thm.~3.1,~3.2]{GS03}.
		
		\item[(iii)] The \textup{conforming Crouzeix--Raviart element} for $d=2$, \textit{i.e.}, $V_h=(\mathbb{P}^2_c(\mathcal{T}_h)\bigoplus\mathbb{B}(\mathcal{T}_h))^2$ and $Q_h=\mathbb{P}^1(\mathcal{T}_h)$, introduced in \cite{CR73}; see also \cite[Ex.\ 8.6.1]{BBF13}. An operator $\Pi_h^V$ satisfying Assumption~\ref{ass:proj-div}(i) is given in \cite[p.\ 49]{CR73} and it can be shown to satisfy Assumption \ref{ass:proj-div}(ii); see, \textit{e.g.}, \cite[Thm.\ 3.3]{GS03}.
		
		\item[(iv)] The {\textup{first order Bernardi--Raugel element}} for $d\in  \{2,3\}$, \textit{i.e.}, 
		$V_h=(\mathbb{P}^1_c(\mathcal{T}_h)\bigoplus\mathbb{B}_{\tiny \mathscr{F}}(\mathcal{T}_h))^d$, where $\mathbb{B}_{\tiny \mathscr{F}}(\mathcal{T}_h)$ is the facet bubble function space, and $Q_h=\mathbb{P}^0(\mathcal{T}_h)$, introduced in \cite[Sec.\ II]{BR85}. For $d=2$ is often  referred to as \textup{reduced $\mathbb{P}^2$-$\mathbb{P}^0$-element} or as \textup{2D SMALL element}; see,  \textit{e.g.},  \cite[Rem.\ 8.4.2]{BBF13} and \cite[Chap.\ II.2.1]{GR86}.  An operator $\Pi_h^V$ satisfying Assumption \ref{ass:proj-div} is given in
		\cite[Sec.\ II.]{BR85}.
		
		\item[(v)] The \textup{second order Bernardi--Raugel element} for $d=3$, introduced in \cite[Sec.\ III]{BR85}; see also \cite[Ex.\ 8.7.2]{BBF13} and \cite[Chap.\ II.2.3]{GR86}. An operator $\Pi_h^V$ satisfying Assumption \ref{ass:proj-div} is given in \cite[Sec.\ III.3]{BR85}.
	\end{itemize}
\end{remark}

\subsection{FE formulations}

\hspace*{5mm}Appealing to \cite{JK23_inhom}, the FE formulation of Problem (\hyperlink{Q}{Q}) reads:

\textit{Problem (Q$_h$).}\hypertarget{Qhfe}{} For given $\bff  \in (L^{p'}(\Omega))^d$, find $(\bfv_h,q_h)^\top \in \Vo_h\times \Qo_h$ such that for every $(\bfz_h,z_h)^\top\in $ \linebreak\hspace*{4.5mm} $\Vo_h\times Q_h$, it holds that\vspace*{-1mm}\enlargethispage{10mm}
\begin{align}
	\label{eq:FE_Qh}
	\begin{aligned}
		(\bfS(\bfD \bfv_h),\bfD \bfz_h)_\Omega +b(\bfv_h, \bfv_h, \bfz_h)_\Omega- (\divo\bfz_h,q_h)_\Omega &=
		(\bff,\bfz_h)_\Omega \,,
		\\
		(\divo\bfv_h,z_h)_\Omega &= 0\,,\\[-1mm]
	\end{aligned}
\end{align}
where \textit{Temam's modification} of the convective term (\textit{cf}.\ \cite{tem}) $b\colon [\Vo]^3\to \mathbb{R}$, for every $(\bfu,\bfw,\bfz)^\top \in [\Vo]^3$, is defined via\vspace*{-1mm}
\begin{align*}
	b(\bfu,\bfw,\bfz)\coloneqq \tfrac{1}{2}(\bfz\otimes \bfu,\nabla \bfw)_{\Omega}-\tfrac{1}{2}(\bfw\otimes\bfu,\nabla \bfz)_{\Omega}\,.\\[-7mm]
\end{align*}
The FE formulation of Problem (\hyperlink{P}{P}) reads:

\textit{Problem (P$_h$).}\hypertarget{Phfe}{} For given $\bff  \in (L^{p'}(\Omega))^d$, find $\bfv_h\in \Vo_h(0)$ such that for every $\bfz_h\in \Vo_h$, it holds that\vspace*{-1mm}
\begin{align}	\label{eq:FE_Ph}
	(\bfS(\bfD \bfv_h),\bfD \bfz_h)_\Omega +b(\bfv_h, \bfv_h, \bfz_h)_\Omega= (\bff,\bfz_h)_\Omega\,,\\[-7mm]\notag
\end{align}
\hspace*{5mm}where\vspace*{-1mm}
\begin{align*}
	\Vo_h(0)\coloneqq \{\bfz_h\in \Vo_h\mid (\divo\bfz_h,z_h)_\Omega=0\text{ for all }z_h\in Q_h\}\,.\\[-7mm]
\end{align*} 
The existence of a weak solution of the problem Problem (\hyperlink{Qhfe}{Q$_h$}) and Problem (\hyperlink{Phfe}{P$_h$}) 
can be established as in the continuous case in two steps: first, the well-posedness of Problem (\hyperlink{Phfe}{P$_h$})  is shown using pseudo-monotone operator theory; second, given the well-posedness of Problem (\hyperlink{Phfe}{P$_h$}), the well-posedness of Problem (\hyperlink{Qhfe}{Q$_h$})  follows using the following discrete inf-sup stability result:

\begin{lemma}\label{lem:ismd}
	Let Assumption \ref{ass:proj-div} be satisfied. Then, for every $q_h\in \Qo_h$, it holds that
	\begin{align*}
		c\,	\|q_h\|_{p',\Omega}\leq \sup_{\bfz_h\in \Vo_h\,:\,\|\nabla\bfz_h\|_{p,\Omega}\leq 1}{(q_h,\divo\bfz_h)_\Omega}\,,
	\end{align*}
	where $c$ depends only on $d$, $m$, $\ell$, $p$, $\kappa_0$, and $\Omega$.
\end{lemma}
\begin{proof}
	See \cite[Lem. 4.1]{bdr-phi-stokes}.
\end{proof}

\newpage
\subsection{A priori error estimates}

\hspace*{5mm}In order to derive an \textit{a priori} error estimate for the pressure, it is first necessary to derive \textit{a priori} error estimates for the velocity vector field. Fortunately, both has already been done recently in the contribution  \cite{JK23_inhom}, whose main results are summarized in the following theorem:

\begin{theorem}
	\label{thm:error_FE}
	Let $\SSS$ satisfy Assumption~\ref{assum:extra_stress} with
	$p\in(2,\infty)$ and $\delta> 0$. Moreover, let $\bff\in (L^{p'}(\Omega))^d$ and 
	$\bfF(\bfD \bfv) \in (W^{1,2}(\Omega))^{d\times d}$.   Then,
	there exists a constant $c_0 >0$, depending on the characteristics of
	$\SSS$, $\delta^{-1}$, $\kappa_0$, $m$, $\ell$, and $\Omega$, such that if $\norm{\nabla
		\bfv}_2\le c_0$, then, it holds that
	\begin{align*}
		\|\bfF(\bfD \bfv_h)-\bfF(\bfD
          \bfv)\|_{2,\Omega}^2+\| q_h-q\|_{p',\Omega}^2
          &\leq c\, h^2\, \big (\|\nabla	\bfF(\bfD \bfv)
            \|_{2,\Omega}^2 + \|\nabla q\|_{p',\Omega}^2\big )+c\,\rho_{(\phi_{\abs{\bfD \bfv}})^*,\Omega}(h\nabla q)\,,
	\end{align*}
	where $c>0$ depends only on the characteristics of
	$\SSS$, $\delta^{-1}$, $\kappa_0$,  $m$, $\ell$, $\Omega$, $\|\bfv	\|_{\infty,\Omega}$,  $\|\bfD\bfv	\|_{p,\Omega}$, and $c_0$.
\end{theorem}\vspace*{-5mm}
\begin{proof}
  The assertion follows from \cite[Theorem 4.2, Theorem
  4.4]{JK23_inhom} and $s(p)=p$ for $p>2$ in \cite{JK23_inhom}.
\end{proof}

\begin{corollary}\label{cor:error_FE}
	Let the assumptions of Theorem \ref{thm:error_FE} be satisfied. Then,~it~holds~that
	\begin{align*}
		\|\bfF(\bfD \bfv_h) - \bfF(\bfD
		\bfv)\|_{2,\Omega}^2 +
		\| q_h-q\|_{p',\Omega}^2\leq  c\, h^2\, \big (\|\nabla	\bfF(\bfD \bfv)
            \|_{2,\Omega}^2 + \|\nabla q\|_{p',\Omega}^2\big )+c\,h^{p'}\rho_{(\phi_{\smash{\vert \bfD\bfv\vert}})^*,\Omega}(\nabla q)\,,
	\end{align*}
	where $c>0$ depends on the characteristics of
	$\SSS$, $\delta^{-1}$, $\kappa_0$,  $m$, $\ell$, $\Omega$, $\|\bfv
	\|_{\infty,\Omega}$,  $\|\bfD\bfv	\|_{p,\Omega}$, and $c_0$. If $\bff \in (L^{2}(\Omega))^d$, then
	\begin{align*}
		\|\bfF(\bfD \bfv_h) - \bfF(\bfD
		\bfv)\|_{2,\Omega}^2 +
		\| q_h-q\|_{p',\Omega}^2\leq  c\, h^2  \,\big( \|\nabla\bfF(\bfD \bfv)
		\|_{2,\Omega}^2 + \|\nabla q\|_{p',\Omega}^2+\|\nabla q\|_{2,\mu_{\bfD\bfv}^{-1},\Omega}^2
		\big)\,,
	\end{align*}
	where $\mu_{\bfD\bfv}\hspace*{-0.175em}\coloneqq \hspace*{-0.175em}(\delta+\vert \overline{\bfD\bfv}\vert)^{p-2}$ and $c\hspace*{-0.175em}>\hspace*{-0.175em}0$ depends on the characteristics of
	$\SSS$, $\delta^{-1}$, $\kappa_0$,  $m$, $\ell$, $\Omega$,~$\|\bfv
	\|_{\infty,\Omega}$,  $\|\bfD\bfv	\|_{p,\Omega}$,~and~$c_0$.
\end{corollary}\vspace*{-5mm}
\begin{proof}
  The assertion follows from \cite[Corollary 4.3, Theorem
  4.4]{JK23_inhom} and $s(p)=p$ for $p>2$ in \cite{JK23_inhom}.
\end{proof}

%

\begin{remark}
	\begin{itemize}[{(ii)}]
		\item[(i)]   In \cite{JK23_inhom}, numerical experiments confirmed the quasi-optimality of the \textit{a priori} error estimates for the velocity vector field  in Corollary \ref{cor:error_FE}.
		
		\item[(ii)]   In \cite{JK23_inhom}, numerical experiments indicated the sub-optimality of the \textit{a priori} error estimates for the kinematic pressure in Corollary \ref{cor:error_FE} in two dimensions.
	\end{itemize} 
\end{remark}

Let us continue with the main result of this paper. It proves the
conjecture from \cite{kr-pnse-ldg-3} under a mild additional
assumption on the velocity vector field, \textit{i.e.}, the Muckenhoupt regularity condition \eqref{eq:reg-assumption}.

\begin{theorem}
	\label{thm:error_pressure_optimal}
	Let the assumptions of Theorem \ref{thm:error_FE} be satisfied and let $\mu_{\bfD\bfv} \coloneqq (\delta+\vert \overline{\bfD\bfv}\vert)^{p-2}\in A_2(\mathbb{R}^d)$.
	Then, it holds that
	\begin{align*}
		\rho_{(\varphi_{\smash{\vert\bfD\bfv\vert}})^*,\Omega}(q_h-q)\leq
          c\,h^2 \,\big (\|\nabla\bfF(\bfD\bfv)\|_{2,\Omega}^2
          +\|\nabla q\|^2_{p',\Omega}\big )
		+c\,\rho_{(\varphi_{\vert \bfD\bfv\vert})^*,\Omega}(h\,\nabla q)+c\,\inf_{z_h\in \Qo_h}{\rho_{(\varphi_{\smash{\vert
						\bfD\bfv\vert}})^*,\Omega}(q-z_h)}\,,
	\end{align*}
	where $c>0$ depends on the characteristics of
	$\SSS$, $\delta^{-1}$, $\kappa_0$,  $m$, $\ell$, $\Omega$, $\|\bfv
	\|_{\infty,\Omega}$,  $\|\bfD\bfv	\|_{p,\Omega}$,  and $c_0$.
\end{theorem}

\begin{corollary}\label{cor:error_pressure_optimal}\allowdisplaybreaks
	Let the assumptions of Theorem \ref{thm:error_pressure_optimal} be satisfied. Then,~it~holds~that
	\begin{align*}
		\rho_{(\varphi_{\smash{\vert\bfD\bfv\vert}})^*,\Omega}(q_h-q)\leq
          c\, h^2\, \big (\|\nabla\bfF(\bfD \bfv)\|_{2,\Omega}^2 +\|\nabla q\|^2_{p',\Omega}\big )+c\,h^{p'}\rho_{(\phi_{\smash{\vert \bfD\bfv\vert}})^*,\Omega}(\nabla q)\,,
	\end{align*}
	where $c>0$ depends on the characteristics of
	$\SSS$, $\delta^{-1}$, $\kappa_0$,  $m$, $\ell$, $\Omega$, $\|\bfv
	\|_{\infty,\Omega}$,  $\|\bfD\bfv	\|_{p,\Omega}$,  and $c_0$.\newpage If $\bff \in (L^{2}(\Omega))^d$,~then
	\begin{align*}
		\rho_{(\varphi_{\smash{\vert\bfD\bfv\vert}})^*,\Omega}(q_h-q)
		\leq  c\, h^2\, \big( \|\nabla\bfF(\bfD \bfv)
		\|_{2,\Omega}^2 +\|\nabla q\|^2_{p',\Omega}+\|\nabla q\|_{2,\mu_{\bfD\bfv}^{-1},\Omega}^2 \big)\,,
	\end{align*}
	where $c>0$ depends on the characteristics of
	$\SSS$, $\delta^{-1}$, $\kappa_0$,  $m$, $\ell$, $\Omega$, $\|\bfv
	\|_{\infty,\Omega}$,  $\|\bfD\bfv	\|_{p,\Omega}$,  and $c_0$.
\end{corollary}

The key ingredient in the proof of Theorem \ref{thm:error_FE} is the following discrete convex~conjugation~inequality.\enlargethispage{10mm}

\begin{lemma}[Discrete convex conjugation inequality]\label{lem:discrete_convex_conjugation_ineq_FE}
	Let  $p\in [2,\infty)$ and $\delta\ge  0$. Moreover, 
	let $\bfA\in (L^p(\Omega))^{d\times d}$  with  $\mu_{\bfA} \coloneqq (\delta+\vert\overline{\bfA}\vert)^{p-2}\in A_2(\mathbb{R}^d)$. Then, there exists a constant $c>0$, depending on $p$,  $\kappa_0$, $m$, $\ell$, $\Omega$, and $[\mu_{\bfA} ]_{A_2(\mathbb{R}^d)}$, such that for every $z_h \in \Qo_h$, it holds that
	\begin{align*}
		\rho_{(\varphi_{\smash{\vert\bfA\vert}})^*,\Omega}(z_h )\leq \sup_{\bfz_h\in \Vo_h}{\big[(z_h ,\divo\bfz_h)_\Omega-\smash{\tfrac{1}{c}}\,\rho_{\varphi_{\smash{\vert\bfA\vert}},\Omega}(\nabla\bfz_h)\big]}\,.\\[-8mm]
	\end{align*}
\end{lemma}

The proof of Lemma \ref{lem:discrete_convex_conjugation_ineq_FE} is based on two key ingredients.
The first key ingredient is the following  continuous counterpart.

\begin{lemma}[Convex conjugation inequality]\label{lem:convex_conjugation_ineq}
	Let   $p\in [2,\infty)$ and $\delta\ge 0$. Moreover, 
	let $\bfA\in (L^p(\Omega))^{d\times d}$ with  $\mu_{\bfA} \coloneqq (\delta+\vert\overline{\bfA}\vert)^{p-2}\in A_2(\mathbb{R}^d)$. Then, there exists a constant $c>0$, depending on $p$, $\Omega$, and the $[\mu_{\bfA} ]_{A_2(\mathbb{R}^d)}$, such that for every $z\in \Qo$, it holds that
	\begin{align*}
		\rho_{(\varphi_{\smash{\vert\bfA\vert}})^*,\Omega}(z)\leq \sup_{\bfz\in \Vo}{\big[(z,\textup{div}\,\bfz)_\Omega-\smash{\tfrac{1}{c}}\,\rho_{\varphi_{\smash{\vert\bfA\vert}},\Omega}(\nabla \bfz)\big]}\,.\\[-8mm]
	\end{align*}
\end{lemma}\vspace*{-5mm}

\begin{proof}
	See \cite[Lem.\ 18]{KR24pressure}.
\end{proof}

The second key ingredient is the following stability result for the projection operators $\{\Pi_h^V\}_{h>0}$ in terms of the shifted modular $\rho_{\varphi_{\smash{\vert \bfD\bfv\vert}},\Omega}$.
	
	\begin{lemma}[Shifted modular estimate for $\Pi_h^{\smash{V}}$]\label{lem:shifted_modular_Pi_div}
		Let  $p\in [2,\infty)$ and $\delta\ge  0$. Moreover, 
		let ${\bfA\in (L^p(\Omega))^{d\times d}}$ with $\mu_{\bfA} \coloneqq (\delta+\vert\overline{\bfA}\vert)^{p-2}\in A_2(\mathbb{R}^d)$.  Then, there exists a constant $c>0$, depending on $k$, $\kappa_0$, $p$, $\Omega$, and $[\mu_{\bfA}]_{A_2(\mathbb{R}^d)}$, such that for every $\bfz\in \Vo$ and $K\in \mathcal{T}_h$, it holds that
		\begin{align*}
			\rho_{\varphi_{\smash{\vert\bfA\vert}},K}(\nabla \Pi_h^{\smash{V}} \bfz)\leq c\, \rho_{\varphi_{\smash{\vert\bfA\vert}},\omega_K}(\nabla \bfz)\,.\\[-8mm]
		\end{align*}
	\end{lemma}\vspace*{-5mm}

\begin{proof}
	Using H\"older's inequality, a local  norm equivalence
        (\textit{cf}.\ \cite[Lem.\ 12.1]{EG21}), the consequence
        \cite[Corollary 4.8]{dr-interpol} of the local $L^1$-$W^{1,1}$-stability of $\Pi_h^{\smash{V}}$ (\textit{cf}.\ Assumption \ref{ass:proj-div}), $\mu_{\bfA}\hspace*{-0.05em}\in\hspace*{-0.05em}
	A_2(\mathbb{R}^d)$, and   $\vert
	B_K\vert\hspace*{-0.05em}\sim\hspace*{-0.05em} \vert K\vert\hspace*{-0.05em}\sim \hspace*{-0.05em}\vert \omega_K\vert$, where  $B_K\hspace*{-0.05em}\coloneqq\hspace*{-0.05em} B_{\textup{diam}(\omega_K)}^d(x_K)$ and  $x_K$ is the barycenter of $K$, we find that
	\begin{align}\label{lem:shifted_modular_Pi_div.1}
		\begin{aligned}
			\|\nabla \Pi_h^{\smash{V}} \bfz\|_{2,\mu_{\bfA},K}&\leq c\, \|\nabla \Pi_h^{\smash{V}}\bfz\|_{\infty,K}\|\mu_{\bfA}\|_{1,K}^{\smash{1/2}}\\&\leq 
			c\,\vert K\vert^{-1}\|\nabla \Pi_h^{\smash{V}} \bfz\|_{1,K} \|\mu_{\bfA}\|_{1,K}^{\smash{1/2}}
			\\&\leq c\,	\vert K\vert^{-1}\|\nabla \bfz\|_{1,\omega_K} \|\mu_{\bfA}\|_{1,K}^{\smash{1/2}}
			\\&\leq c\,	\vert K\vert^{-1}\|\nabla \bfz\mu_{\bfA}^{\smash{1/2}}\mu_{\bfA}^{\smash{-1/2}}\|_{1,\omega_K} \|\mu_{\bfA}\|_{1,K}^{\smash{1/2}}
			\\&\leq c\, \|\nabla \bfz\|_{2,\mu_{\bfA},\omega_K}\|\vert \omega_K\vert^{-1}\mu_{\bfA}^{-1}\|_{1,\omega_K}^{\smash{1/2}} \|\vert K\vert^{-1}\mu_{\bfA}\|_{1,K}^{\smash{1/2}}
			\\&\leq c\,	\|\nabla \bfz\|_{2,\mu_{\bfA},\omega_K}\smash{\|\vert B_K\vert^{-1}\mu_{\bfA}^{-1}\|_{1,B_K}^{\smash{1/2}} \|\vert B_K\vert^{-1}\mu_{\bfA}\|_{1,B_K}^{\smash{1/2}}}
			\\&\leq c\, \|\nabla \bfz\|_{2,\mu_{\bfA},\omega_K}\,.
		\end{aligned}
	\end{align}
	On the other hand, appealing to \cite[Cor.\ 4.8]{dr-interpol}, for every $K\in \mathcal{T}_h$, we have that
	\begin{align}\label{lem:shifted_modular_Pi_div.2}
		\|\nabla \Pi_h^{\smash{V}} \bfz\|_{p,K}\leq c\, \|\nabla\bfz\|_{p,\omega_K}\,.
	\end{align}
	Due to $p\ge 2$, for every $t\ge 0$ and a.e.\ $x\in \Omega$, it holds that
	\begin{align}
		\begin{aligned}
			\smash{(\delta + \abs{\bfA(x)}+ t)^{p-2} t^2\leq 2^{p-2}\,( \mu_{\bfA}(x) \,t^2 +\, t^p)\leq 2^{p-2}\, (\delta + \abs{\bfA(x)} + t)^{p-2} t^2\,,}
		\end{aligned}\label{lem:key.3}
	\end{align}
	which, appealing to $\varphi_{\smash{\vert\bfA(x)\vert}}(t)\sim (\delta + \abs{\bfA(x)}+ t)^{p-2} t^2$ uniformly with respect to all $t\ge 0$ and a.e.\ $x\in \Omega$,
	proves that $L^p_0(\Omega)\cap L^2_0(\Omega;\mu )=L^{\varphi_{\smash{\vert\bfA\vert}}}_0(\Omega)$. 
	Therefore, due to \eqref{lem:key.3}, for every $K\in \mathcal{T}_h$, from \eqref{lem:shifted_modular_Pi_div.1} and \eqref{lem:shifted_modular_Pi_div.2}, we conclude that
	\begin{align*}
		\rho_{\varphi_{\smash{\vert\bfA\vert}},K}(\nabla \Pi_h^{\smash{V}} \bfz)&\leq c\,\big(	\|\nabla \Pi_h^{\smash{V}} \bfz\|_{2,\mu_{\bfA},K}^2+	\|\nabla \Pi_h^{\smash{V}} \bfz\|_{p,K}^p\big)\\&\leq c\,\big(\|\nabla \bfz\|_{2,\mu_{\bfA},\omega_K}^2+\|\nabla\bfz\|_{p,\omega_K}^p\big) \\&\leq c\,\rho_{\varphi_{\smash{\vert\bfA\vert}},\omega_K}(\nabla \bfz)\,,
	\end{align*}
	which is the claimed stability estimate.
\end{proof}
Eventually, having Lemma \ref{lem:shifted_modular_Pi_div} and Lemma \ref{lem:convex_conjugation_ineq} at hand, we are in the position to prove Lemma~\ref{lem:discrete_convex_conjugation_ineq_FE}.

\begin{proof}[Proof (of Lemma
  \ref{lem:discrete_convex_conjugation_ineq_FE}).] Using Lemma
  \ref{lem:convex_conjugation_ineq}, \eqref{eq:div_preserving}
  (\textit{cf}.\ Assumption \ref{ass:proj-div} (i)), and Lemma \ref{lem:shifted_modular_Pi_div}, we find that
	\begin{align*}
		\rho_{(\varphi_{\smash{\vert\bfA\vert}})^*,\Omega}(z_h )&\leq \sup_{\bfz\in  \Vo}{\big[(z_h ,\textup{div}\,\bfz)_\Omega-\smash{\tfrac{1}{c}}\,\rho_{\varphi_{\smash{\vert\bfA\vert}},\Omega}(\nabla \bfz)\big]}
		\\&= \sup_{\bfz\in  \Vo}{\big[(z_h ,\divo\Pi_h^{\smash{V}}\bfz)_\Omega-\smash{\tfrac{1}{c}}\,\rho_{\varphi_{\smash{\vert\bfA\vert}},\Omega}(\nabla \bfz)\big]}
		\\&\leq \sup_{\bfz\in  \Vo}{\big[(z_h ,\divo\Pi_h^{\smash{V}}\bfz)_\Omega-\smash{\tfrac{1}{c}}\,\rho_{\varphi_{\smash{\vert\bfA\vert}},\Omega}(\nabla \Pi_h^{\smash{V}} \bfz)\big]}
		\\&\leq \sup_{\bfz_h\in \Vo_h}{\big[(z_h ,\divo\bfz_h)_\Omega-\smash{\tfrac{1}{c}}\,\rho_{\varphi_{\smash{\vert\bfA\vert}},\Omega}(\nabla \bfz_h)\big]}\,,
	\end{align*}
	which is the claimed discrete convex conjugation formula.
\end{proof}

Before we can move on to the proof of Theorem \ref{thm:error_pressure_optimal}, we need the following approximation property result for the operators $\{\Pi_h^Q\}_{h>0}$ in terms of the shifted modular $\rho_{(\phi_{\smash{\vert\bfD \bfv\vert }})^*,\Omega}$.

\begin{lemma}\label{lem:shifted_Q_approx} Let $p\in [2,\infty)$ and $\delta\ge 0$. Moreover, 
	let ${\bfA\in (L^p(\Omega))^{d\times d}}$ with  $\mu_{\bfA} \coloneqq (\delta+\vert\overline{\bfA}\vert)^{p-2}\in A_2(\mathbb{R}^d)$. 
	Then, there exists a constant $c>0$, depending on $p$, $\ell$, $\Omega$, $\kappa_0$, and $[\mu_{\bfA}]_{A_2(\mathbb{R}^d)}$, such that for every $q\in \Qo$, it holds that
	\begin{align*}
		\rho_{(\phi_{\smash{\vert\bfA\vert }})^*,\Omega}(q-\Pi_h^Qq-\langle \Pi_h^Q q\rangle_\Omega)\leq \begin{cases}
			c\,h^{p'}\rho_{\phi^*,\Omega}(\nabla q)&\text{ if }q\in W^{1,p'}(\Omega)\,,\\
			c\,h^2\,\|\nabla q\|_{\mu_{\bfA}^{-1},2,\Omega}^2&\text{ if } \nabla q\in (L^2(\Omega;\mu_{\bfA}^{-1}))^d\,.
		\end{cases}
	\end{align*}
	In particular, we have that
	\begin{align*}
		\inf_{z_h\in \Qo_h}{	\rho_{(\phi_{\smash{\vert\bfA\vert }})^*,\Omega}(q-z_h)}\leq \begin{cases}
			c\,h^{p'}\rho_{\phi^*,\Omega}(\nabla q)&\text{ if }q\in W^{1,p'}(\Omega)\,,\\
			c\,h^2\,\|\nabla q\|_{\mu_{\bfA}^{-1},2,\Omega}^2&\text{ if }  \nabla q\in (L^2(\Omega;\mu_{\bfA}^{-1}))^d\,.
		\end{cases}
	\end{align*}
\end{lemma}\vspace*{-5mm}

\begin{proof}
		See \cite[Lem.\  23]{KR24pressure}.
\end{proof}\newpage

\begin{proof}[Proof (of Theorem \ref{thm:error_pressure_optimal}).] 
	Appealing to Lemma \ref{lem:discrete_convex_conjugation_ineq_FE}, there exists a constant $c>0$ such
	that for every $z_h\in \Qo_h$, it holds that  
	\begin{align}
		\label{thm:error_pressure_optimal.1}
		\rho_{(\varphi_{\smash{\vert\bfD\bfv\vert}})^*,\Omega}(z_h)\leq \sup_{\bfz_h\in \Vo_h}{\big[(z_h,\divo\bfz_h)_\Omega-\smash{\tfrac{1}{c}}\,\rho_{\varphi_{\smash{\vert\bfD\bfv\vert}},\Omega}(\nabla\bfz_h)\big]}\,.
	\end{align}
	On the other hand, in view of \eqref{eq:FE_Qh}, \eqref{eq:q1}
        we have, for every $\bfz_h\in \Vo_h$, that
	\begin{align}\label{thm:error_pressure_optimal.2}
		\begin{aligned}
			(q_h-q,\divo \bfz_h)_\Omega
			&=(\SSS(\bfD \bfv_h) - \SSS(\bfD
			\bfv),\bfD\bfz_h)_\Omega
			+\big[b(\bfv_h,\bfv_h,\bfz_h)-b(\bfv,\bfv,\bfz_h)\big]
			\\&\eqqcolon I_h^1+ I_h^2\,.
		\end{aligned}
	\end{align}		
	So, let us next estimate $I_h^1$ and $I_h^2$ for arbitrary $\bfz_h\in \Vo_h$:
	
	\textit{ad $I_h^1$}.
	Using the $\varepsilon$-Young inequality \eqref{ineq:young} with $\psi=\varphi_{\smash{\vert\bfD\bfv\vert}}$, \eqref{eq:hammera}, and Theorem \ref{thm:error_FE}, we find that
	\begin{align}\label{thm:error_pressure_optimal.3}
		\begin{aligned}
			\vert I_h^1\vert &\leq c_\varepsilon\,\rho_{(\varphi_{\smash{\vert\bfD\bfv\vert}})^*,\Omega}(\bfS(\bfD\bfv_h)-\bfS(\bfD\bfv))+\varepsilon\,\rho_{\varphi_{\smash{\vert\bfD\bfv\vert}},\Omega}(\bfD\bfz_h)\\&\leq
			c_\varepsilon\,\norm{\bfF(\bfD\bfv)
				-\bfF(\bfD\bfv_h)}_2^2+\varepsilon\,\rho_{\varphi_{\smash{\vert\bfD\bfv\vert}},\Omega}(\nabla\bfz_h)
			\,.
			\\&\leq
			c_\varepsilon\,\big(h^2 \,  (\|\nabla	\bfF(\bfD \bfv)
            \|_{2,\Omega}^2 + \|\nabla q\|_{p',\Omega}^2 )
			+c\,\rho_{(\varphi_{\vert \bfD\bfv\vert})^*,\Omega}(h\,\nabla q)\big)+\varepsilon\,\rho_{\varphi_{\smash{\vert\bfD\bfv\vert}},\Omega}(\nabla\bfz_h)
			\,.
		\end{aligned}
	\end{align}
	
	\textit{ad $I_h^2$}. Abbreviating $\bfe_h\coloneqq \bfv_h-\bfv\in \Vo$, 
	there holds the decomposition
	\begin{align}\label{thm:error_pressure_optimal.4}
		\begin{aligned}
		I_h^2 &=b(\bfv,\bfv-\Pi_h^V\bfv,\Pi_h^V\bfz_h)+b(\bfe_h,\Pi_h^V\bfv,\bfz_h)+b(\bfv_h,\Pi_h^V\bfe_h,\bfz_h)\\&\eqqcolon I_h^{2,1}+ I_h^{2,2}+ I_h^{2,3}\,,
	\end{aligned}
	\end{align}
	so that it is left to estimate $I_h^{2,1}$, $I_h^{2,2}$,  and $I_h^{2,3}$:
	
	\textit{ad $I_h^{2,1}$}. The definition of $b\colon [\Vo]^3\to \mathbb{R}$ yields that
	\begin{align*}
		I_h^{2,1}=(\bfz_h\otimes \bfv,\nabla (\bfv-\Pi_h^V\bfv))_\Omega+((\bfv-\Pi_h^V\bfv)\otimes \bfv,\nabla \bfz_h)_\Omega\,.
	\end{align*}
	Thus, using Hölder's,
	Poincar\'e's, and  Korn's inequality,  the  $\varepsilon$-Young inequality \eqref{ineq:young} with $\psi=\vert\cdot\vert^2$,  that $\vert \bfA-\bfB\vert^2\leq \varphi_{\smash{\vert \bfA\vert}}(\vert \bfA-\bfB\vert)\sim \vert \bfF(\bfA)-\bfF(\bfB)\vert^2$  and $\vert \bfA\vert^2\leq c\, \varphi_{\vert \bfB\vert}(\vert\bfA\vert)$ uniformly with respect to $\bfA,\bfB\in \mathbb{R}^{d\times d}$ (due to $p>2$), and the approximation properties of $\Pi_h^V$ (\textit{cf}.\ \cite[Thms. 3.4, 5.1]{bdr-phi-stokes}), we find that
	\begin{align}
		\label{thm:error_pressure_optimal.5}
		\begin{aligned}
		\vert I_h^{2,1}\vert &\leq 2\,\|\bfv\|_{\infty,\Omega}\|\bfv-\Pi_h^V\bfv\|_{1,2,\Omega}\|\bfz_h\|_{1,2,\Omega}
		\\&\leq c_\varepsilon\,\|\bfv\|_{\infty,\Omega}^2\|\bfD\bfv-\bfD\Pi_h^V\bfv\|_{2,\Omega}^2+\varepsilon\,c\,\|\nabla\bfz_h\|_{2,\Omega}^2
			\\&\leq c_\varepsilon\,\|\bfF(\bfD\bfv)-\bfF(\bfD\Pi_h^V\bfv)\|_{2,\Omega}^2+\varepsilon\,c\,\rho_{\varphi_{\smash{\vert \bfD\bfv\vert}},\Omega}(\nabla\bfz_h)
				\\&\leq c_\varepsilon\,h^2\,\|\nabla\bfF(\bfD\bfv)\|_{2,\Omega}^2+\varepsilon\,c\,\rho_{\varphi_{\smash{\vert \bfD\bfv\vert}},\Omega}(\nabla\bfz_h)\,.
					\end{aligned}
	\end{align}

	\textit{ad $I_h^{2,2}$}. The definition of $b\colon [\Vo]^3\to \mathbb{R}$  yields that
	\begin{align*}
		I_h^{2,2}=(\bfz_h\otimes \bfe_h,\nabla\Pi_h^V\bfv)_\Omega+(\Pi_h^V\bfv\otimes \bfe_h,\nabla \bfz_h)_\Omega\,.
	\end{align*}
	Thus, using Hölder's, Sobolev's,  and Korn's inequality, the stability properties of $\Pi_h^V$ (\textit{cf}.\ \cite[Cor.~4.8]{dr-interpol}), the  $\varepsilon$-Young inequality \eqref{ineq:young} with $\psi=\vert\cdot\vert^2$, 
 that $\vert \bfA-\bfB\vert^2\leq \varphi_{\smash{\vert \bfA\vert}}(\vert \bfA-\bfB\vert)\sim \vert \bfF(\bfA)-\bfF(\bfB)\vert^2$  and $\vert \bfA\vert^2\leq c\, \varphi_{\vert \bfB\vert}(\vert\bfA\vert)$ uniformly with respect to all $\bfA,\bfB\in \mathbb{R}^{d\times d}$ (due to $p>2$), and \eqref{eq:hammera},~we~find~that 
	\begin{align}
			\label{thm:error_pressure_optimal.6}
		\begin{aligned}
		\vert I_h^{2,2}\vert &\leq (\|\Pi_h^V\bfv\|_{4,\Omega}+\|\nabla\Pi_h^V\bfv\|_{2,\Omega})\|\bfe_h\|_{4,\Omega}(\|\bfz_h\|_{4,\Omega}+\|\nabla\bfz_h\|_{2,\Omega})
		\\&\leq c\,\|\nabla\bfv\|_{2,\Omega}\|\bfD\bfe_h\|_{2,\Omega} \|\nabla\bfz_h\|_{2,\Omega}
		\\&\leq c_\varepsilon\,c_0^2\,\|\bfD\bfe_h\|_{2,\Omega}^2+\varepsilon\,c\, \|\nabla\bfz_h\|_{2,\Omega}^2
		\\&\leq c_\varepsilon\,\|\bfF(\bfD\bfv_h)-\bfF(\bfD\bfv)\|_{2,\Omega}^2+\varepsilon\,c\,\rho_{\varphi_{\smash{\vert \bfD\bfv\vert}},\Omega}(\nabla\bfz_h)
			\\&\leq c_\varepsilon\,\big (h^2 \,  (\|\nabla	\bfF(\bfD \bfv)
            \|_{2,\Omega}^2 + \|\nabla q\|_{p',\Omega}^2
            )+\rho_{(\phi_{\smash{\vert\bfD \bfv\vert
                }})^*,\Omega}(h\,\nabla q)\big )+\varepsilon\,c\,\rho_{\varphi_{\smash{\vert \bfD\bfv\vert}},\Omega}(\nabla\bfz_h)\,.
			\end{aligned}
	\end{align}
	
		\textit{ad $I_h^{2,3}$}. The definition of $b\colon [\Vo]^3\to \mathbb{R}$  yields that
	\begin{align*}
	\smash{	I_h^{2,3}=(\bfz_h\otimes \bfv_h,\nabla\Pi_h^V\bfe_h)_\Omega+(\Pi_h^V\bfe_h\otimes \bfv_h,\nabla \bfz_h)_\Omega\,.}
	\end{align*}
	Thus, using Hölder's, Sobolev's, and Korn's inequality the stability properties of $\Pi_h^V$ (\textit{cf}.\ \cite[Cor.~4.8]{dr-interpol}), the  $\varepsilon$-Young inequality \eqref{ineq:young} with $\psi=\vert\cdot\vert^2$, 
and  that $\vert \bfA-\bfB\vert^2\leq \varphi_{\smash{\vert \bfA\vert}}(\vert \bfA-\bfB\vert)\sim \vert \bfF(\bfA)-\bfF(\bfB)\vert^2$ and $\vert \bfA\vert^2\leq c\, \varphi_{\vert \bfB\vert}(\vert\bfA\vert)$  uniformly with respect to all $\bfA,\bfB\in \mathbb{R}^{d\times d}$ (due~to~$p>2$), we find that 
	\begin{align}
		\label{thm:error_pressure_optimal.7}
		\begin{aligned}
			\vert I_h^{2,3}\vert &\leq \|\bfv_h\|_{4,\Omega}(\|\Pi_h^V\bfe_h\|_{4,\Omega}+\|\nabla\Pi_h^V\bfe_h\|_{2,\Omega})(\|\bfz_h\|_{4,\Omega}+\|\nabla\bfz_h\|_{2,\Omega})
			\\&\leq c\,\|\nabla\bfv_h\|_{2,\Omega}\|\bfD\bfe_h\|_{2,\Omega} \|\nabla\bfz_h\|_{2,\Omega}
			\\&\leq c_\varepsilon\,\|\bfD\bfe_h\|_{2,\Omega}^2+\varepsilon\,c\, \|\nabla\bfz_h\|_{2,\Omega}^2
			\\&\leq c_\varepsilon\,\|\bfF(\bfD\bfv_h)-\bfF(\bfD\bfv)\|_{2,\Omega}^2+\varepsilon\,c\,\rho_{\varphi_{\smash{\vert \bfD\bfv\vert}},\Omega}(\nabla\bfz_h)
			\\&\leq c_\varepsilon\,\big (h^2 \,  (\|\nabla	\bfF(\bfD \bfv)
            \|_{2,\Omega}^2 + \|\nabla q\|_{p',\Omega}^2 )+\rho_{(\phi_{\smash{\vert\bfD \bfv\vert }})^*,\Omega}(h\,\nabla q)\big)+\varepsilon\,c\,\rho_{\varphi_{\smash{\vert \bfD\bfv\vert}},\Omega}(\nabla\bfz_h)\,.
		\end{aligned}
	\end{align}	
	Eventually, combining \eqref{thm:error_pressure_optimal.5}--\eqref{thm:error_pressure_optimal.7} in \eqref{thm:error_pressure_optimal.4}, we arrive at
	\begin{align}\label{thm:error_pressure_optimal.8}
		\vert I_h^2\vert \leq c_\varepsilon\,\big (h^2 \,  (\|\nabla	\bfF(\bfD \bfv)
            \|_{2,\Omega}^2 + \|\nabla q\|_{p',\Omega}^2 )+\rho_{(\phi_{\smash{\vert\bfD \bfv\vert }})^*,\Omega}(h\,\nabla q)\big)+\varepsilon\, c\,\rho_{\varphi_{\smash{\vert \bfD\bfv\vert}},\Omega}(\nabla\bfz_h)\,.
	\end{align}
	Putting it all together, from \eqref{thm:error_pressure_optimal.3} and  \eqref{thm:error_pressure_optimal.8} in \eqref{thm:error_pressure_optimal.2}, for every $\bfz_h\in \Vo_h$, we conclude that
	\begin{align}\label{thm:error_pressure_optimal.11}
		(q_h-q,\divo \bfz_h)_\Omega\leq  c_\varepsilon\,\big (h^2 \,  (\|\nabla	\bfF(\bfD \bfv)
            \|_{2,\Omega}^2 + \|\nabla q\|_{p',\Omega}^2 )+\rho_{(\phi_{\smash{\vert\bfD \bfv\vert }})^*,\Omega}(h\,\nabla q)\big)+\varepsilon\, c\,\rho_{\varphi_{\smash{\vert \bfD\bfv\vert}},\Omega}(\nabla\bfz_h)\,.
	\end{align}
	Therefore, using Lemma
        \ref{lem:discrete_convex_conjugation_ineq_FE} and
        \eqref{thm:error_pressure_optimal.11} for sufficiently small $\varepsilon>0$, once more, the $\varepsilon$-Young's inequality \eqref{ineq:young} with $\psi = \varphi_{\smash{\vert
			\bfD\bfv\vert}}$,
	for sufficiently small $\varepsilon>0$,  for every ${z_h\in \Qo_h}$, we find that
	\begin{align}\label{eq:fast_pressure_optimal}
          \begin{aligned}
            \rho_{(\varphi_{\smash{\vert\bfD\bfv\vert}})^*,\Omega}(q_h-q)
            &\leq c\,\rho_{(\varphi_{\smash{\vert
                  \bfD\bfv\vert}})^*,\Omega}(q_h-z_h)+c\,
            \rho_{(\varphi_{\smash{\vert
                  \bfD\bfv\vert}})^*,\Omega}(q-z_h)
            \\
            &\leq c\, \sup_{\bfz_h\in
              \Vo_h}{\big[(q_h-z_h,\divo\bfz_h)_\Omega-\smash{\tfrac{1}{c}}\,\rho_{\varphi_{\smash{\vert
                    \bfD\bfv\vert}},\Omega}(\nabla
              \bfz_h)\big]}+c\,\rho_{(\varphi_{\smash{\vert
                  \bfD\bfv\vert}})^*,\Omega}(q-z_h)
            \\[-1mm]
            &\leq c\, \sup_{\bfz_h\in
              \Vo_h}{\big[(q-z_h,\divo\bfz_h)_\Omega-\smash{\tfrac{1}{c}}\,\rho_{\varphi_{\smash{\vert
                    \bfD\bfv\vert}},\Omega}(\nabla
              \bfz_h)\big]}+c\,\rho_{(\varphi_{\smash{\vert
                  \bfD\bfv\vert}})^*,\Omega}(q-z_h)
            \\[-1mm]
            &\quad +c\,\big (h^2 \,  (\|\nabla	\bfF(\bfD \bfv)
            \|_{2,\Omega}^2 + \|\nabla q\|_{p',\Omega}^2 )+\rho_{(\phi_{\smash{\vert\bfD \bfv\vert }})^*,\Omega}(h\,\nabla q)\big)
            \\
            &\leq c\,h^2 \,\big (\|\nabla\bfF(\bfD\bfv)\|_{2,\Omega}^2 + \|\nabla q\|_{p',\Omega}^2 \big)
            +c\,\rho_{(\varphi_{\smash{\vert
                  \bfD\bfv\vert}})^*,\Omega}(h\,\nabla
            q)+c\,\rho_{(\varphi_{\smash{\vert
                  \bfD\bfv\vert}})^*,\Omega}(q-z_h)\,.\hspace{-10mm}
          \end{aligned}
	\end{align}
	Eventually, taking in \eqref{eq:fast_pressure_optimal} the infimum with respect to ${z_h\in \Qo_h}$,  
	we conclude the proof.
\end{proof}\vspace*{-2.5mm}

\begin{proof}[Proof (of Corollary \ref{cor:error_pressure_optimal}).]\enlargethispage{7mm}
	Follows from Theorem \ref{thm:error_pressure_optimal} by means
        of Lemma \ref{lem:shifted_Q_approx} and the properties of $(\varphi_{\smash{\vert
                  \bfD\bfv\vert}})^*$, namely $(\varphi_a)^*(h\,t)\hspace{-0.1em} \le \varphi^*(h\,t)\hspace{-0.1em} \le\hspace{-0.1em}
    c\, h^{p'} \varphi^*(t)$ for all
    $t,a\hspace{-0.1em}\ge\hspace{-0.1em} 0$,
    valid~for~${p\hspace{-0.1em}>\hspace{-0.1em}2}$
    (cf.~\cite{bdr-phi-stokes}) and $(\varphi_a)^*(h\,t)\hspace{-0.1em} \sim \big( (\delta+a)^{p-1}
  +h\,t\big )^{p'-2}\, h^2\, t^2 \hspace{-0.1em}
  \le\hspace{-0.1em}(\delta+a)^{2-p}\, h^2\, t^2 $ for all
  $t,a\hspace{-0.1em}\ge\hspace{-0.1em} 0$, since $p>2$.
\end{proof}\newpage

\section{Numerical experiments}\label{sec:experiments} 

\hspace*{5mm}In this section, we complement the theoretical findings of  Section \ref{sec:fe} via numerical experiments:~first, we carry out numerical experiments to confirm the quasi-optimality of the \textit{a priori} error estimates in Corollary \ref{cor:error_pressure_optimal} with respect to the Muckenhoupt regularity condition~\eqref{eq:reg-assumption}; second, we carry out numerical experiments to examine the Muckenhoupt regularity condition \eqref{eq:reg-assumption} for its sharpness.~Before we do so, we first give some implementation details.\enlargethispage{7.5mm} 

\subsection{Implementation details} 
\hspace*{5mm}All experiments were carried out using the finite element software~\mbox{\texttt{FEniCS}} (version 2019.1.0, \textit{cf}.~\cite{LW10}). 
All graphics were generated with the help of the \texttt{Matplotlib} library (version 3.5.1,~\textit{cf}.~\cite{Hun07}). 
In the numerical experiments addressing the FE formulation Problem (\hyperlink{Qhfe}{Q$_h$}) (or Problem (\hyperlink{Phfe}{P$_h$}), resp.),
we deploy both the MINI element (\textit{cf}.\ Remark \ref{FEM.V}(i)) and the Taylor--Hood element (\textit{cf}.\ Remark \ref{FEM.V}(ii)).

We approximate discrete solutions $(\bfv_h,q_h)^{\top}\in \Vo_h\times \Qo_h$ of Problem (\hyperlink{Qhfe}{Q$_h$}) using the Newton solver from \mbox{\texttt{PETSc}} (version 3.17.3, \textit{cf}.\ \cite{LW10}), with absolute tolerance of $\tau_{abs}= 1\textrm{e}{-}8$ and relative tolerance~of~${\tau_{rel}=1\textrm{e}{-}10}$. The linear system emerging in each Newton iteration is solved using a sparse direct solver from \texttt{MUMPS} (version~5.5.0,~\textit{cf}.~\cite{mumps}).  In the implementation, the uniqueness of the pressure is enforced via adding a zero mean condition. Each integral that does contain non-discrete functions (\textit{e.g.}, the right-hand side integral in Problem (\hyperlink{Qhfe}{Q$_h$}) and Problem  (\hyperlink{Phfe}{P$_h$})) or the error quantities \eqref{error_quantities}) is discretized using a quadrature rule by G.\ Strang and G.\ Fix (\textit{cf}.\ \cite{strang_fix}) with degree of precision 6 (\textit{i.e.}, 12 quadrature points on each triangle) in two dimensions (\textit{i.e.}, $d=2$) and  using
the Keast rule \texttt{(KEAST7)} (\textit{cf}.\ \cite{keast}) with degree of precision 6 (\textit{i.e.}, employing
24 quadrature points on each element) in three dimensions (\textit{i.e.},~$d=3$).\vspace*{-0.5mm}
\subsubsection{\itshape General experimental set-up}\vspace*{-0.5mm}

\hspace*{5mm}For  $\Omega=(0,1)^d$ we employ Problem (\hyperlink{Qhfe}{Q$_h$}) (or \mbox{Problem} (\hyperlink{Phfe}{P$_h$}), resp.) to approximate the system~\eqref{eq:p-navier-stokes} with  $\bfS\colon  \mathbb{R}^{d\times d}\to\smash{\mathbb{R}^{d\times d}_{\mathrm{sym}}}$, for  every $\bfA\in\mathbb{R}^{d\times d}$~defined~via
\begin{align*}
	\bfS(\bfA) \coloneqq \mu_0\,(\delta+\vert \bfA^{\textup{sym}}\vert)^{p-2}\bfA^{\textup{sym}}\,,
\end{align*}  
where  $\delta\coloneqq 1.0\times 10^{-5}$, $\mu_0 \coloneqq \frac{1}{2}$, and $ p\in [2.25, 3.5]$.

We construct an initial triangulation $\pazocal{T}_{h_0}$, where $h_0=1$, for $d\in \{2,3\}$, in the following way: 

\begin{itemize}[{($d=3$)}]
	\item[($d=2$)] We subdivide $\Omega=(0,1)^2$ into four triangles along its diagonals.
	\item[($d=3$)] We subdivide $\Omega=(0,1)^3$ into six tetrahedron forming a Kuhn triangulation $\mathcal{T}_0$ (\textit{cf}.\ \cite{kuhn}). 
\end{itemize}

Then, finer triangulations $\pazocal{T}_{h_i}$, $i=1,\dots,6$, where $h_{i+1}=\frac{h_i}{2}$ for all $i=0,\dots,5$, are 
obtained by  applying the red-refinement rule (\textit{cf}.\ \cite[Def.~4.8(i)]{Ba16}). 

Then, for $i = 0,\ldots,6$, we compute discrete solutions 
$(\bfv_i,q_i)^{\top} \coloneqq  (\bfv_{h_i},q_{h_i})^{\top}\in \smash{\Vo_{h_i}\times  \Qo_{h_i}}$ of Problem (\hyperlink{Qhfe}{Q$_{h_i}$}) 
and the error quantities 
\begin{align}\label{error_quantities} 
	\left.\begin{aligned}
		e_{\bfF,i}&\coloneqq \|\bfF(\bfD\bfv_i)-\bfF(\bfD\bfv)\|_{2,\Omega}\,,\\
		e_{q,i}^{\textup{norm}}&\coloneqq \|q_i-q\|_{p',\Omega}\,,\\[-0.5mm]
		e_{q,i}^{\textup{modular}}&\coloneqq (\rho_{(\varphi_{\smash{\vert\bfD\bfv\vert}})^*,\Omega}(q_i-q))^{\smash{1/2}}\,,
	\end{aligned}\quad\right\}\,,\quad i=0,\dots,6\,.
\end{align}

As  estimation  of  the  convergence rates, we compute the experimental order of convergence~(EOC)
\begin{align}
	\texttt{EOC}_i(e_i)\coloneqq\frac{\log(e_i/e_{i-1})}{\log(h_i/h_{i-1})}\,, \quad i=1,\dots,6\,,\label{eoc}
\end{align}
where for every $i= 0,\dots,6$, we denote by $e_i$ either $ \smash{e_{q,i}^{\textup{norm}}}$ or $\smash{e_{q,i}^{\textup{modular}}}$.

\subsubsection{\itshape Quasi-optimality of the \textit{a priori} error estimates derived in Corollary \ref{cor:error_pressure_optimal}}

\hspace*{5mm}To confirm the quasi-optimality of the \textit{a priori} error estimates derived in Corollary~\ref{cor:error_pressure_optimal},  we restrict  the two-dimensional case and choose the right-hand side $\bff \in (L^{p'}(\Omega))^2$ and the Dirichlet boundary data $\bfv_0\in (W^{1-1/p,p}(\partial\Omega))^2$  such that $\bfv\in V$ and $q \in \Qo$, for every $x\coloneqq (x_1,x_2)^\top\in \Omega$ defined via\vspace*{-0.5mm}
\begin{align}
	\bfv(x)\coloneqq \vert x\vert^{\beta} (x_2,-x_1)^\top\,, \qquad q(x)\coloneqq  \vert x\vert^{\gamma}-\langle\,\vert \cdot\vert^{\gamma}\,\rangle_\Omega\,,
\end{align}
are a solutions to  \eqref{eq:p-navier-stokes}. Since, in this case, we consider given inhomogeneous Dirichlet boundary data, we do not use Problem (\hyperlink{Qhfe}{Q$_h$}) (or Problem (\hyperlink{Phfe}{P$_h$}), resp.), but the consistent inhomogeneous generalization; \textit{i.e.}, we employ the inhomogeneous~generalization~of Problem (\hyperlink{Qhfe}{Q$_h$}) (or Problem (\hyperlink{Phfe}{P$_h$}), resp.) from \cite{JK23_inhom}.

Concerning the regularity of the velocity vector field, we choose $\beta=0.01$, which precisely yields that ${\bfF(\bfD\bfv)\in (W^{1,2}(\Omega))^{2\times 2}}$. Concerning the regularity of the pressure, we distinguish two~different~cases:
\begin{itemize}[{(Case 2)}]
	\item[\hypertarget{case_1}{(Case 1)}] We choose $\gamma= 1-2/p'+1.0\times 10^{-2}$, which just yields that $\smash{q \in
		W^{1,p'}(\Omega)}$;
	\item[\hypertarget{case_2}{(Case 2)}] We choose $\gamma= \beta(p-2)/2+1.0\times 10^{-2}$, which  just yields that $\nabla q \in
	(L^2(\Omega;\smash{\mu_{\bfD\bfv}^{-1}}))^2$.
\end{itemize} 
In both Case \hyperlink{case_1}{1} and Case \hyperlink{case_2}{2}, we have that\footnote{as to be expected in the two-dimensional case (\textit{cf}.\ Remark \ref{rem:muckenhoupt}(ii)).\vspace*{-5mm}} $\bfD\bfv\hspace*{-0.1em}\in\hspace*{-0.1em} (C^{0,\beta}(\overline{\Omega}))^{2\times 2}$ and, thus, ${\mu_{\bfD\bfv}\hspace*{-0.1em}\coloneqq\hspace*{-0.1em} (\delta+\vert \overline{\bfD\bfv}\vert)^{p-2}\hspace*{-0.1em}\in\hspace*{-0.1em} A_2(\mathbb{R}^2)}$~(\textit{cf}.\ Remark \ref{rem:muckenhoupt}(i)).
Hence, in Case \hyperlink{case_1}{1}, we can expect the
convergence rate $\texttt{EOC}_i(e_i^{\textup{modular}})\approx p'/2$, $i=1,\ldots,6$, while in Case \hyperlink{case_2}{2},  we can expect the convergence rate $\texttt{EOC}_i(e_i^{\textup{modular}})\approx 1$,~$i=1,\ldots,6$~(\textit{cf}.~Corollary~\ref{cor:error_pressure_optimal}).

For different values of $p\in \{2.25, 2.5, 2.75, 3, 3.25, 3.5\}$ and a
series of triangulations~$\mathcal{T}_{h_i}$, $i = 1,\dots,6$,
obtained by regular, global refinement as described above, the EOC is
computed and presented in Table~\ref{tab2}, and
Table~\ref{tab3}, 
respectively. In both Case \hyperlink{case_1}{1} and Case \hyperlink{case_2}{2}, 
we observe the expected a convergence rate of about $\texttt{EOC}_i(e_i^{\textup{modular}})\approx p'/2$, $i=1,\ldots,6$, (in Case \hyperlink{case_1}{1}) and ${\texttt{EOC}_i(e_i^{\textup{modular}})\approx 1}$,~${i=1,\dots, 6}$, (in Case \hyperlink{case_2}{2}).\enlargethispage{7mm}

\begin{table}[ht]
	\setlength\tabcolsep{3pt}
	\centering
	\begin{tabular}{c |c|c|c|c|c|c|c|c|c|c|c|c|} \cmidrule(){1-13}
		\multicolumn{1}{|c||}{\cellcolor{lightgray}$\gamma$}	
		& \multicolumn{6}{c||}{\cellcolor{lightgray}Case \hyperlink{case_1}{1}}   & \multicolumn{6}{c|}{\cellcolor{lightgray}Case \hyperlink{case_2}{2}}\\ 
		\hline 
		
		\multicolumn{1}{|c||}{\cellcolor{lightgray}\diagbox[height=1.1\line,width=0.11\dimexpr\linewidth]{\vspace{-0.6mm}$i$}{\\[-5mm] $p$}}
		& \cellcolor{lightgray}2.25 & \cellcolor{lightgray}2.5  & \cellcolor{lightgray}2.75  &  \cellcolor{lightgray}3.0 & \cellcolor{lightgray}3.25  & \multicolumn{1}{c||}{\cellcolor{lightgray}3.5} &  \multicolumn{1}{c|}{\cellcolor{lightgray}2.25}   & \cellcolor{lightgray}2.5  & \cellcolor{lightgray}2.75  & \cellcolor{lightgray}3.0  & \cellcolor{lightgray}3.25 &   \cellcolor{lightgray}3.5 \\ \hline\hline
		\multicolumn{1}{|c||}{\cellcolor{lightgray}$1$}             & 0.875 & 0.812 & 0.770 & 0.741 & 0.718 & \multicolumn{1}{c||}{0.699} & \multicolumn{1}{c|}{0.966} & 0.979 & 1.116 & 1.281 & 1.371 & 1.376 \\ \hline
		\multicolumn{1}{|c||}{\cellcolor{lightgray}$2$}             & 0.904 & 0.836 & 0.787 & 0.750 & 0.721 & \multicolumn{1}{c||}{0.699} & \multicolumn{1}{c|}{1.006} & 1.011 & 1.027 & 1.096 & 1.200 & 1.300 \\ \hline
		\multicolumn{1}{|c||}{\cellcolor{lightgray}$3$}             & 0.907 & 0.840 & 0.791 & 0.755 & 0.727 & \multicolumn{1}{c||}{0.704} & \multicolumn{1}{c|}{1.009} & 1.010 & 1.013 & 1.020 & 1.058 & 1.113 \\ \hline
		\multicolumn{1}{|c||}{\cellcolor{lightgray}$4$}             & 0.908 & 0.841 & 0.793 & 0.757 & 0.728 & \multicolumn{1}{c||}{0.706} & \multicolumn{1}{c|}{1.010} & 1.010 & 1.011 & 1.014 & 1.017 & 1.051 \\ \hline
		\multicolumn{1}{|c||}{\cellcolor{lightgray}$5$}             & 0.908 & 0.841 & 0.793 & 0.757 & 0.729 & \multicolumn{1}{c||}{0.707} & \multicolumn{1}{c|}{1.010} & 1.010 & 1.011 & 1.011 & 1.014 & 1.016 \\ \hline
		\multicolumn{1}{|c||}{\cellcolor{lightgray}$6$}             & 0.908 & 0.841 & 0.793 & 0.757 & 0.729 & \multicolumn{1}{c||}{0.707} & \multicolumn{1}{c|}{1.010} & 1.010 & 1.011 & 1.011 & 1.012 & 1.015 \\ \hline\hline
		\multicolumn{1}{|c||}{\cellcolor{lightgray} theory}         & 0.900 & 0.833 & 0.786 & 0.750 & 0.722 & \multicolumn{1}{c||}{0.700} & \multicolumn{1}{c|}{1.000} & 1.000 & 1.000 & 1.000 & 1.000 & 1.000 \\ \hline
	\end{tabular}
	\caption{Experimental order of convergence (MINI): $\textup{\texttt{EOC}}_i(e_{q,i}	^{\textup{modular}})$,~${i=1,\dots,6}$.}\label{tab2}\vspace*{-11mm}
\end{table}

\begin{table}[ht]
	\setlength\tabcolsep{3pt}
	\centering
	\begin{tabular}{c |c|c|c|c|c|c|c|c|c|c|c|c|} \cmidrule(){1-13}
		\multicolumn{1}{|c||}{\cellcolor{lightgray}$\gamma$}	
		& \multicolumn{6}{c||}{\cellcolor{lightgray}Case \hyperlink{case_1}{1}}   & \multicolumn{6}{c|}{\cellcolor{lightgray}Case \hyperlink{case_2}{2}}\\ 
		\hline 
		
		\multicolumn{1}{|c||}{\cellcolor{lightgray}\diagbox[height=1.1\line,width=0.11\dimexpr\linewidth]{\vspace{-0.6mm}$i$}{\\[-5mm] $p$}}
		& \cellcolor{lightgray}2.25 & \cellcolor{lightgray}2.5  & \cellcolor{lightgray}2.75  &  \cellcolor{lightgray}3.0 & \cellcolor{lightgray}3.25  & \multicolumn{1}{c||}{\cellcolor{lightgray}3.5} &  \multicolumn{1}{c|}{\cellcolor{lightgray}2.25}   & \cellcolor{lightgray}2.5  & \cellcolor{lightgray}2.75  & \cellcolor{lightgray}3.0  & \cellcolor{lightgray}3.25 &   \cellcolor{lightgray}3.5 \\ \hline\hline
		\multicolumn{1}{|c||}{\cellcolor{lightgray}$1$}             & 0.898 & 0.807 & 0.750 & 0.709 & 0.677 & \multicolumn{1}{c||}{0.652} & \multicolumn{1}{c|}{1.265} & 1.061 & 0.966 & 0.946 & 0.950 & 0.954 \\ \hline
		\multicolumn{1}{|c||}{\cellcolor{lightgray}$2$}             & 0.912 & 0.838 & 0.789 & 0.752 & 0.723 & \multicolumn{1}{c||}{0.700} & \multicolumn{1}{c|}{1.121} & 1.045 & 1.032 & 1.025 & 1.010 & 1.011 \\ \hline
		\multicolumn{1}{|c||}{\cellcolor{lightgray}$3$}             & 0.908 & 0.840 & 0.791 & 0.754 & 0.726 & \multicolumn{1}{c||}{0.703} & \multicolumn{1}{c|}{1.014} & 1.014 & 1.016 & 1.019 & 1.021 & 1.022 \\ \hline
		\multicolumn{1}{|c||}{\cellcolor{lightgray}$4$}             & 0.909 & 0.841 & 0.793 & 0.757 & 0.728 & \multicolumn{1}{c||}{0.706} & \multicolumn{1}{c|}{1.006} & 1.010 & 1.012 & 1.012 & 1.014 & 1.014 \\ \hline
		\multicolumn{1}{|c||}{\cellcolor{lightgray}$5$}             & 0.909 & 0.841 & 0.793 & 0.757 & 0.729 & \multicolumn{1}{c||}{0.707} & \multicolumn{1}{c|}{1.008} & 1.010 & 1.011 & 1.011 & 1.012 & 1.014 \\ \hline
		\multicolumn{1}{|c||}{\cellcolor{lightgray}$6$}             & 0.909 & 0.841 & 0.793 & 0.757 & 0.729 & \multicolumn{1}{c||}{0.707} & \multicolumn{1}{c|}{1.009} & 1.010 & 1.011 & 1.011 & 1.012 & 1.013 \\ \hline\hline
		\multicolumn{1}{|c||}{\cellcolor{lightgray} theory}         & 0.900 & 0.833 & 0.786 & 0.750 & 0.722 & \multicolumn{1}{c||}{0.700} & \multicolumn{1}{c|}{1.000} & 1.000 & 1.000 & 1.000 & 1.000 & 1.000 \\ \hline
	\end{tabular}
	\caption{Experimental order of convergence (Taylor--Hood): $\textup{\texttt{EOC}}_i(e_{q,i}^{\textup{modular}})$,~${i=1,\dots,6}$.}\label{tab3}\vspace*{-5mm}
\end{table}

\subsubsection{\textit{(Non-)sharpness of the Muckenhoupt  regularity condition \eqref{eq:reg-assumption} in  Corollary \ref{cor:error_pressure_optimal}}}

\hspace*{5mm}
In order to examine the Muckenhoupt  regularity condition \eqref{eq:reg-assumption} in  Corollary \ref{cor:error_pressure_optimal} for its sharpness, \textit{i.e.}, necessity, following the construction in \cite{kr-nekorn,kr-nekorn-add}, we consider the three dimensional case
(\textit{i.e.}, $d=3$) and choose the right-hand side ${\bff\in (L^{p'}(\Omega))^3}$ such that $\bfv \in \Vo(0)$ and $q\in \Qo$, for every $x\coloneqq (x_1,x_2,x_3)^\top\in \Omega$ defined via
\begin{align*}
	\bfv(x)\coloneqq \sum_{k\in \mathbb{N}}{\bfv^k(x)}\,,\qquad q(x)\coloneqq 25\,\big(\vert x\vert^\gamma -\langle \vert \cdot\vert^{\gamma} \rangle_\Omega\big)\,,
\end{align*}
for every $k\in \mathbb{N}$, 
\begin{align*}
	\bfv^k(x)\coloneqq \begin{cases}
		\frac{k}{r^k}\big(\frac{r^k}{2}+\frac{\vert x-\mathbf{m}^k\vert }{2}-\vert x-\mathbf{m}^k\vert\big)(x_2-m_2^k,-(x_1-m_1^k),0)^\top&\text{ if }x\in B_{r^k}^3(\mathbf{m}^k)\,,\\
		\mathbf{0}&\text{ if }x\in \Omega\setminus B_{r^k}^3(\mathbf{m}^k)\,.
	\end{cases}
\end{align*}
for sequences $(r^k)_{k\in \mathbb{N}}\subseteq \mathbb{R}^{>0}$ and $(\mathbf{m}^k)_{k\in \mathbb{N}}\coloneqq ((m^k_1,m^k_2,m^k_3)^\top)_{k\in \mathbb{N}}\subseteq \Omega$ satisfying 
\begin{align}\label{conditions}
	r^{k+1}\leq \frac{1}{2}r^k\,,\quad B_{r^k}^3(\mathbf{m}^k)\subseteq \Omega\,,\quad B_{r^k}^3(\mathbf{m}^k)\cap B_{r^{k+1}}^3(\mathbf{m}^{k+1})=\emptyset\qquad\text{ for all }k\in \mathbb{N}\,.
\end{align} 
Following the argumentation in \cite{kr-nekorn,kr-nekorn-add}, one finds that $	\bfF(\bfD\bfv)\in (W^{1,2}(\Omega))^{3\times 3}$ 
while, on the other hand, one finds that
\begin{align}\label{violation}
	\mu_{\bfD\bfv} \coloneqq (\delta+\vert \overline{\bfD\bfv}\vert)^{p-2}\notin A_2(\mathbb{R}^3)\,.
\end{align}
In the numerical experiments, we choose $r^k\coloneqq 2^{-k-2}$ and $\mathbf{m}^k\coloneqq \mathbf{e}_1+3 r^k ( \mathbf{q}_0-\mathbf{e}_1)/\vert \mathbf{q}_0-\mathbf{e}_1\vert$~for~all~${k\in \mathbb{N}}$, where $\mathbf{q}_0\hspace*{-0.15em}\in \hspace*{-0.1em}K_0\hspace*{-0.15em}\coloneqq\hspace*{-0.15em} \textup{conv}\{\mathbf{0},\mathbf{e}_1,\mathbf{e}_1+\mathbf{e}_2,\mathbf{e}_1+\mathbf{e}_2+\mathbf{e}_3\}\hspace*{-0.15em}\in\hspace*{-0.15em} \mathcal{T}_0$ is the quadrature point of the  Keast~rule~\texttt{(KEAST7)} (\textit{cf}.\ \cite{keast}) in the initial triangulation $\mathcal{T}_0$ that is closest to the unit vector $\mathbf{e}_1\coloneqq (1,0,0)^\top\in \mathbb{S}^2$~(\textit{cf}. Figure~\ref{fig:1} (LEFT))\footnote{Note that there is a second quadrature point of the  Keast rule \texttt{(KEAST7)} (\textit{cf}.\ \cite{keast}) in an adjacent element that has the same distance to the first unit vector $\mathbf{e}_1\in \mathbb{S}^2$}. In this way, we can guarantee that conditions \eqref{conditions} are met as well as that for every $i=1,\ldots,5$, there exists an integer $N_i\in \mathbb{N}$ such that 
\begin{align*}
	\mathbf{q}_i\coloneqq (1-2^{-N_i})\,\mathbf{e}_1+2^{-N_i}\, \mathbf{q}_0\in \overline{B_{r^{N_i}}^3(\mathbf{m}^{N_i})}\,,
\end{align*}
which is the quadrature point of the  Keast rule \texttt{(KEAST7)} (\textit{cf}.\ \cite{keast}) in $\mathcal{T}_i$ that is closest to the unit vector $\mathbf{e}_1\in \mathbb{S}^2$. 
This has the consequence that for every $k\in \mathbb{N}$ with $k\ge N_i+1$, the ball $B_{r^k}^3(\mathbf{m}^k)$ does not contain any quadrature point of the Keast rule \texttt{(KEAST7)} (\textit{cf}.\ \cite{keast}) in $\mathcal{T}_i$. Hence,~in~the~numerical~experiments, for every refinement step $i=1,\ldots,5$,~we~can~interchange $	\bfv\coloneqq \sum_{k\in \mathbb{N}}{\bfv^k}\in \smash{(W^{1,p}_0(\Omega))^3}$ (which cannot be directly implemented) by 	$\overline{\bfv}^{\smash{N_i}}\coloneqq \sum_{k=1}^{\smash{N_i}}{\bfv^k}\in \smash{(W^{1,p}_0(\Omega))^3}$ (\textit{cf}.\ Figure \ref{fig:1} and Figure \ref{fig:2}(MIDDLE/RIGHT))
(which can directly be implemented); thus,~we~can interchange  $\bff\coloneqq -\textup{div}\,\bfS(\bfD\bfv)+[\nabla\bfv]\bfv+\nabla q\in (L^{p'}(\Omega))^3$ by $\overline{\bff}^{\smash{N_i}}\coloneqq  -\textup{div}\,\bfS(\bfD\bfv^{\smash{N_i}})+[\nabla\bfv^{\smash{N_i}}]\bfv^{\smash{N_i}}+\nabla q\in (L^{p'}(\Omega))^3$ without changing the discrete formulations. 
Also note that the violation of the Muckenhoupt condition \eqref{violation} can be asymptotically confirmed 
using the Keast rule \texttt{(KEAST7)} (\textit{cf}.\
\cite{keast}). In fact, in Figure \ref{fig:3}, denoting by
$\mathrm{d}\mu_{h_i}$ the discrete measure representing the Keast rule
\texttt{(KEAST7)} (\textit{cf}.\ \cite{keast}) in $\mathcal{T}_i$, it
is indicated that 
\begin{align*}
	E_i\coloneqq \bigg(\int_{B_{r^{N_i}}^3(\mathbf{m}^{N_i})}{\mu_{\bfD\bfv}\,\mathrm{d}\mu_{h_i}}\bigg)\bigg(\int_{B_{r^{N_i}}^3(\mathbf{m}^{N_i})}{\mu_{\bfD\bfv}^{-1}\,\mathrm{d}\mu_{h_i}}\bigg)\to\infty \quad(i\to \infty)\,.
\end{align*}
This in turn indicates the violation of the Muckenhoupt condition \eqref{violation} is sufficiently resolved by the Keast rule \texttt{(KEAST7)} (\textit{cf}.\ \cite{keast}).

\hspace*{-6mm}\begin{minipage}{0.5\textwidth}\vspace*{-2.5mm}
	\begin{figure}[H]
		\hspace*{-0.25cm}\includegraphics[width=7.75cm]{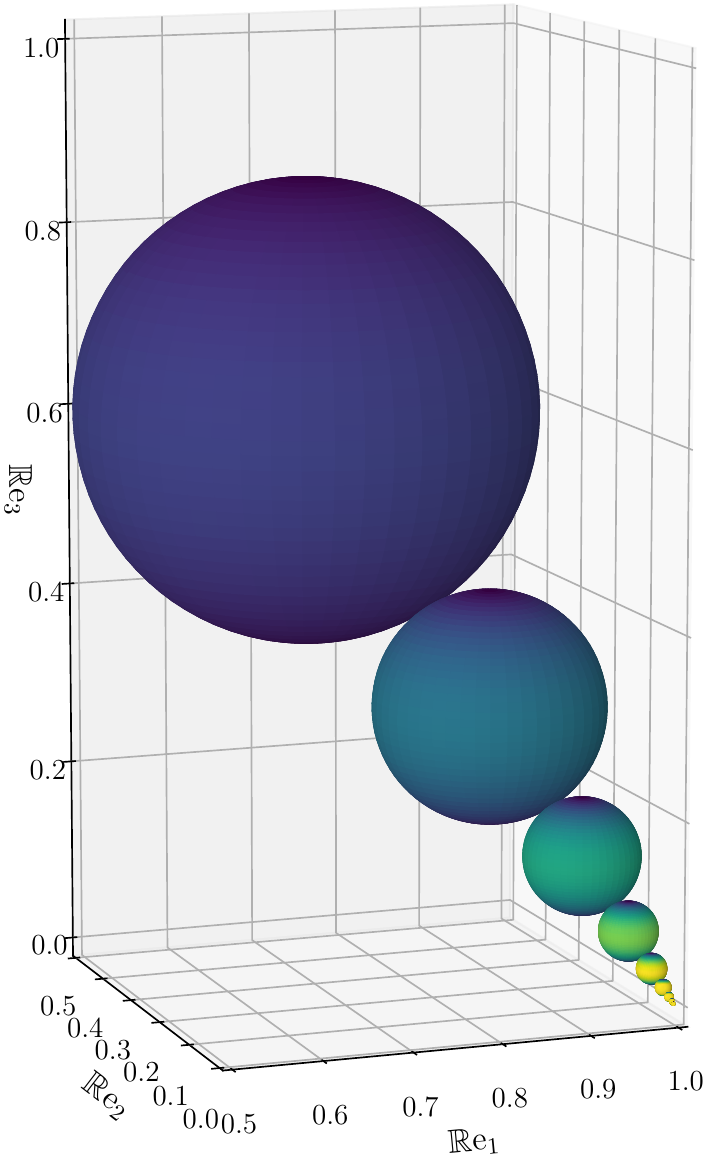}
		\caption{Plot of  the strain rate $\vert\bfD \bfv\vert\colon \Omega\to \mathbb{R}_{\ge 0}$ restricted to the boundary of its support, \textit{i.e.},
			$\partial(\textup{supp}(\vert\bfD \bfv\vert))=\bigcup_{k=1}^\infty{\partial B_{r^k}^3(\mathbf{m}^k)}$.  The color map indicates that the strain rate 
			increases when approaching the first unit vector $\mathbf{e}_1\in \mathbb{S}^2$.}\label{fig:1}
	\end{figure}
\end{minipage}\hspace*{2.5mm}
\begin{minipage}{0.5\textwidth}
	\begin{figure}[H]
		\hspace*{1.5mm}\includegraphics[width=7cm]{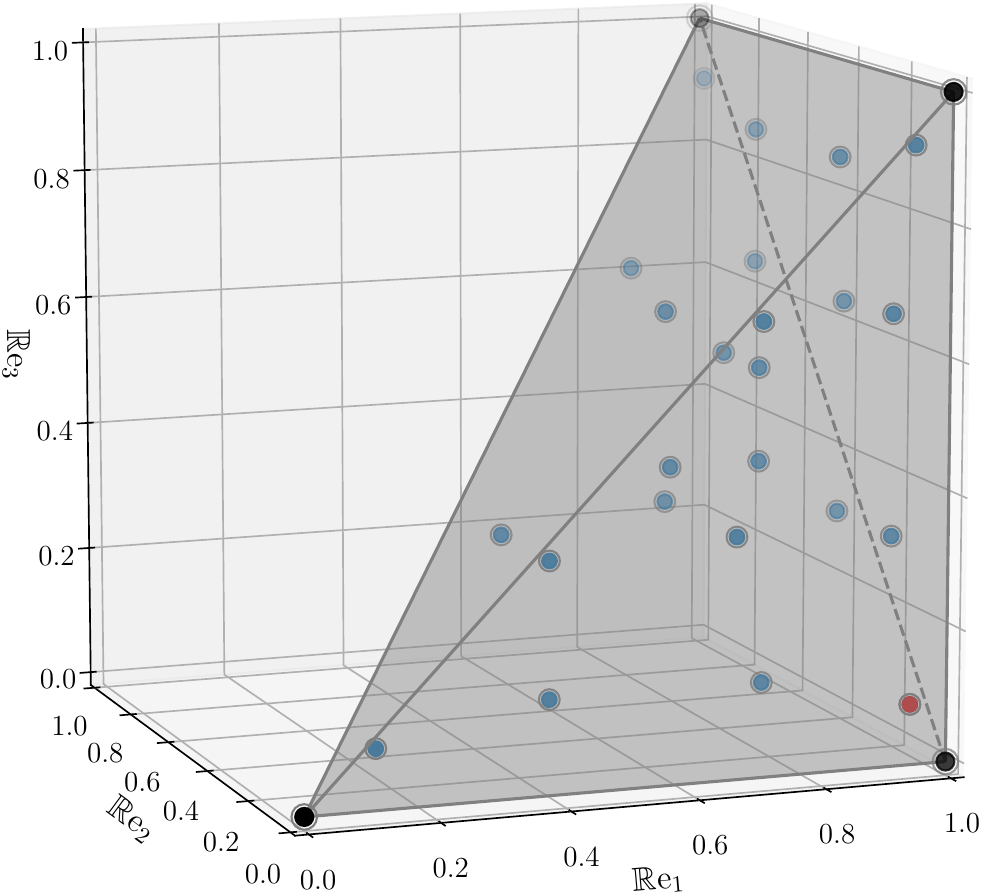}\\\hspace*{1.5mm}\includegraphics[width=7cm]{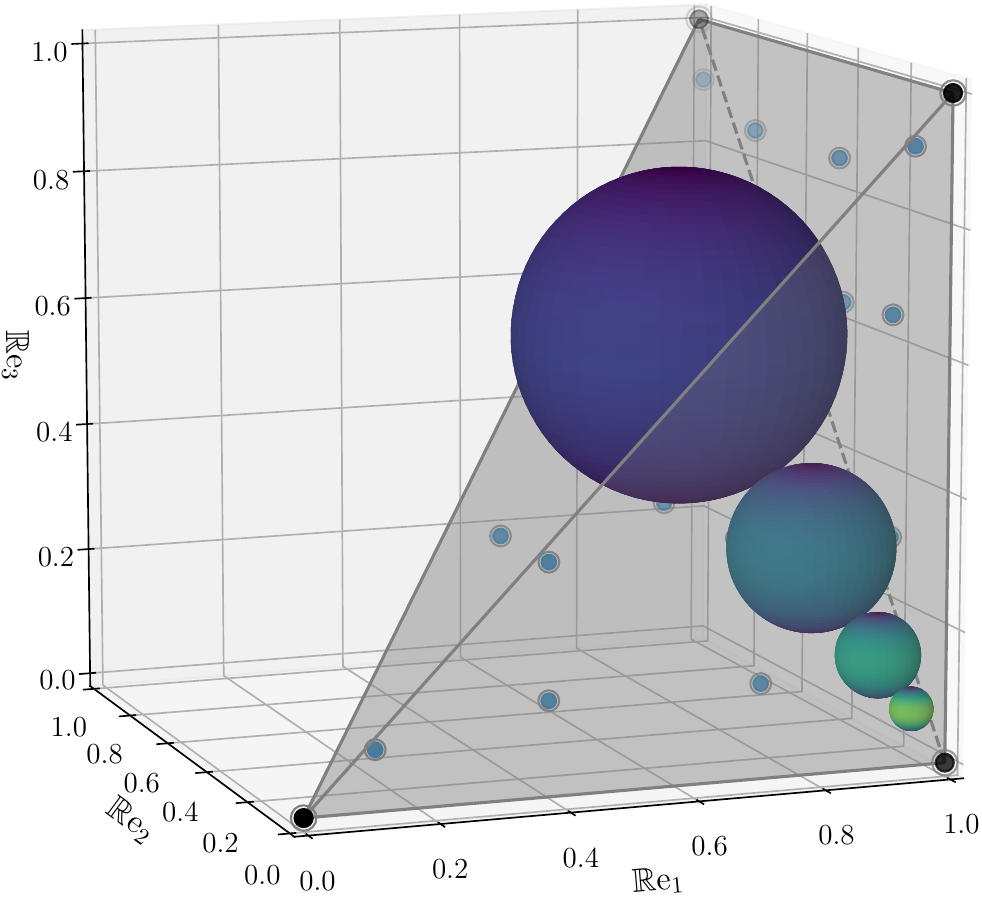}
		\caption{TOP: tetrahedron $K_0\coloneqq \textup{conv}\{\mathbf{0},\mathbf{e}_1,\mathbf{e}_1+\mathbf{e}_2,\mathbf{e}_1+\mathbf{e}_2+\mathbf{e}_3\}$  and the transformed 24 quadrature points of the  Keast rule \texttt{(KEAST7)} (\textit{cf}.\ \cite{keast}). The quadrature point $\mathbf{q}_0\in K_0$ closest to the first unit vector is marked in red; BOTTOM:  $\overline{\bfv}^{N_0}\coloneqq \sum_{k=1}^{\smash{N_0}}{\bfv^k}\in \smash{(W^{1,p}_0(\Omega))^3}$, where $N_0=4$ is minimal such that
			$\mathbf{q}_0\in \textup{supp}(\overline{\bfv}^{N_0})$.}\label{fig:2}
	\end{figure}
\end{minipage}\vspace*{5mm}\enlargethispage{11mm}

For different values of $p\in \{2.25, 2.5, 2.75, 3, 3.25, 3.5\}$ and a
series of triangulations~$\pazocal{T}_{h_i}$, $i = 1,\ldots,4$,
obtained \hspace*{-0.1mm}by \hspace*{-0.1mm}regular, \hspace*{-0.1mm}global \hspace*{-0.1mm}refinement \hspace*{-0.1mm}as \hspace*{-0.1mm}described \hspace*{-0.1mm}above, \hspace*{-0.1mm}the \hspace*{-0.1mm}EOC \hspace*{-0.1mm}is~\hspace*{-0.1mm}computed~\hspace*{-0.1mm}and~\hspace*{-0.1mm}presented~\hspace*{-0.1mm}in~\hspace*{-0.1mm}Tables~\hspace*{-0.1mm}\mbox{\ref{tab4.1}--\ref{tab5.2}}. In both Case \hyperlink{case_1}{1} and Case \hyperlink{case_2}{2}, 
we observe the expected convergence rate of about $\texttt{EOC}_i(e_{\bfF,i})\approx p'/2$, $i=1,2,3$,   (in Case \hyperlink{case_1}{1}) and $\texttt{EOC}_i(e_{\bfF,i})\approx 1$, $ i= 1,\ldots,4$, (in Case \hyperlink{case_2}{2}). On the other hand, in both Case \hyperlink{case_1}{1} and Case \hyperlink{case_2}{2}, although the Muckenhoupt regularity~condition~\eqref{eq:reg-assumption}~is~not~met~(\textit{cf}.~\eqref{violation}), we 
report the convergence rate of about $\texttt{EOC}_i(e_{q,i}^{\textup{modular}})\hspace*{-0.1em}\approx \hspace*{-0.1em} p'/2$, $i= 1,\ldots,4$,   (in Case \hyperlink{case_1}{1})~and~$\texttt{EOC}_i(e_{q,i}^{\textup{modular}})\hspace*{-0.1em}\approx\hspace*{-0.1em} 1$, $ i= 1,\ldots,4$, (in Case \hyperlink{case_2}{2}). 
This shows that the Muckenhoupt regularity condition \eqref{eq:reg-assumption} is only sufficient, but not necessary, for the convergence rates in~Corollary~\ref{cor:error_pressure_optimal}.~In~addition,~in~both~Case~\hyperlink{case_1}{1}~and~Case~\hyperlink{case_2}{2}, we observe the increased convergence rate of about $\texttt{EOC}_i(e_{q,i}^{\textup{norm}})\approx 1$, $i= 1,\ldots,4$,   (in Case \hyperlink{case_1}{1}) and $\texttt{EOC}_i(e_{q,i}^{\textup{norm}})\approx 2/p'$, $ i= 1,\ldots,4$, (in Case \hyperlink{case_2}{2}). This indicates that the convergence rates for the kinematic pressure derived in Corollary \ref{cor:error_pressure_optimal} are also potentially sub-optimal in three dimensions.\newpage

\begin{figure}[ht]
	\hspace*{-2mm}\includegraphics[width=15cm]{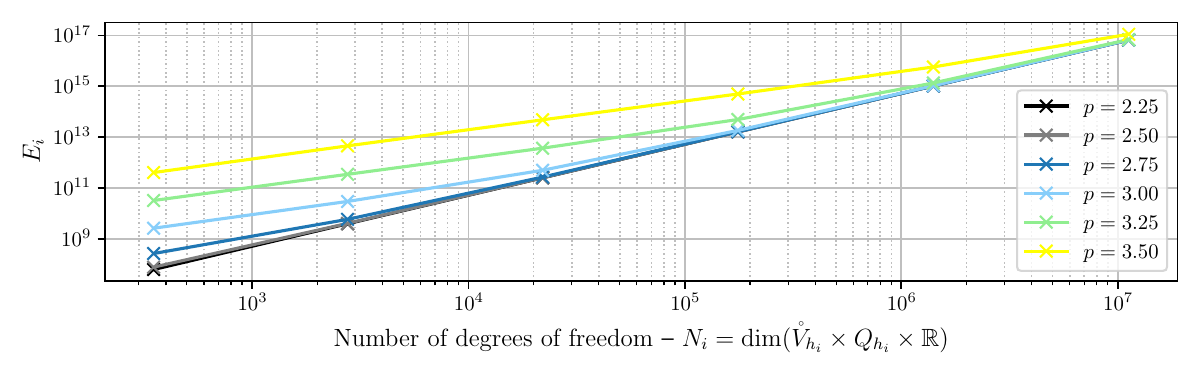}
	\caption{Plots of $E_i\coloneqq (\int_{B_{r^{N_i}}^3(\mathbf{m}^{N_i})}{\mu_{\bfD\bfv}\,\mathrm{d}\mu_{h_i}})(\int_{B_{r^{N_i}}^3(\mathbf{m}^{N_i})}{\mu_{\bfD\bfv}^{-1}\,\mathrm{d}\mu_{h_i}})$, $i=1,\ldots,6$, for $p\in \{2.25,2.5,2.75,3.0,3.25,2.5\}$, where $\mathrm{d}\mu_{h_i}$, $i=1,\ldots,6$, denote the discrete measures representing the Keast rule \texttt{(KEAST7)} (\textit{cf}.\ \cite{keast}), indicating  that $E_i\to \infty$ $(i\to \infty)$~and,~thus, that the violation of the Muckenhoupt condition \eqref{violation} is sufficiently resolved by the Keast rule \texttt{(KEAST7)} (\textit{cf}.\ \cite{keast}).}\label{fig:3}
\end{figure}\vspace*{-5mm}

\begin{table}[ht]
	\setlength\tabcolsep{3pt}
	\centering
	\begin{tabular}{c |c|c|c|c|c|c|c|c|c|c|c|c|} \cmidrule(){1-13}
		\multicolumn{1}{|c||}{\cellcolor{lightgray}$\gamma$}	
		& \multicolumn{6}{c||}{\cellcolor{lightgray}Case \hyperlink{case_1}{1}}   & \multicolumn{6}{c|}{\cellcolor{lightgray}Case \hyperlink{case_2}{2}}\\ 
		\hline 
		
		\multicolumn{1}{|c||}{\cellcolor{lightgray}\diagbox[height=1.1\line,width=0.11\dimexpr\linewidth]{\vspace{-0.6mm}$i$}{\\[-5mm] $p$}}
		& \cellcolor{lightgray}2.25 & \cellcolor{lightgray}2.5  & \cellcolor{lightgray}2.75  &  \cellcolor{lightgray}3.0 & \cellcolor{lightgray}3.25  & \multicolumn{1}{c||}{\cellcolor{lightgray}3.5} &  \multicolumn{1}{c|}{\cellcolor{lightgray}2.25}   & \cellcolor{lightgray}2.5  & \cellcolor{lightgray}2.75  & \cellcolor{lightgray}3.0  & \cellcolor{lightgray}3.25 &   \cellcolor{lightgray}3.5 \\ \hline\hline
		\multicolumn{1}{|c||}{\cellcolor{lightgray}$1$}             & 0.718 & 0.707 & 0.700 & 0.696 & 0.692 & \multicolumn{1}{c||}{0.690} & \multicolumn{1}{c|}{0.798} & 0.834 & 0.857 & 0.874 & 0.888 & 0.899 \\ \hline
		\multicolumn{1}{|c||}{\cellcolor{lightgray}$2$}             & 1.179 & 1.185 & 1.187 & 1.184 & 1.179 & \multicolumn{1}{c||}{1.173} & \multicolumn{1}{c|}{1.379} & 1.533 & 1.642 & 1.720 & 1.776 & 1.817 \\ \hline
		\multicolumn{1}{|c||}{\cellcolor{lightgray}$3$}             & 1.008 & 1.006 & 1.003 & 1.001 & 0.999 & \multicolumn{1}{c||}{0.996} & \multicolumn{1}{c|}{1.170} & 1.289 & 1.377 & 1.443 & 1.493 & 1.533 \\ \hline
		\multicolumn{1}{|c||}{\cellcolor{lightgray}$4$}             & 1.008 & 1.007 & 1.006 & 1.005 & 1.003 & \multicolumn{1}{c||}{1.002} & \multicolumn{1}{c|}{1.172} & 1.297 & 1.392 & 1.465 & 1.522 & 1.566 \\ \hline
		\multicolumn{1}{|c||}{\cellcolor{lightgray}$5$}             & 1.010 & 1.009 & 1.009 & 1.008 & 1.008 & \multicolumn{1}{c||}{1.007} & \multicolumn{1}{c|}{1.175} & 1.304 & 1.405 & 1.484 & 1.546 & 1.595 \\ \hline\hline
		\multicolumn{1}{|c||}{\cellcolor{lightgray} theory}         & 0.900 & 0.833 & 0.786 & 0.750 & 0.722 & \multicolumn{1}{c||}{0.700} & \multicolumn{1}{c|}{1.000} & 1.000 & 1.000 & 1.000 & 1.000 & 1.000 \\ \hline
	\end{tabular}
	\caption{Experimental order of convergence (MINI): $\texttt{EOC}_i(e_{q,i}^{\textup{norm}})$,~${i=1,\dots,5}$.}\vspace*{3mm}
	\label{tab4.1}
	\setlength\tabcolsep{3pt}
	\centering
	\begin{tabular}{c |c|c|c|c|c|c|c|c|c|c|c|c|} \cmidrule(){1-13}
		\multicolumn{1}{|c||}{\cellcolor{lightgray}$\gamma$}	
		& \multicolumn{6}{c||}{\cellcolor{lightgray}Case \hyperlink{case_1}{1}}   & \multicolumn{6}{c|}{\cellcolor{lightgray}Case \hyperlink{case_2}{2}}\\ 
		\hline 
		
		\multicolumn{1}{|c||}{\cellcolor{lightgray}\diagbox[height=1.1\line,width=0.11\dimexpr\linewidth]{\vspace{-0.6mm}$i$}{\\[-5mm] $p$}}
		& \cellcolor{lightgray}2.25 & \cellcolor{lightgray}2.5  & \cellcolor{lightgray}2.75  &  \cellcolor{lightgray}3.0 & \cellcolor{lightgray}3.25  & \multicolumn{1}{c||}{\cellcolor{lightgray}3.5} &  \multicolumn{1}{c|}{\cellcolor{lightgray}2.25}   & \cellcolor{lightgray}2.5  & \cellcolor{lightgray}2.75  & \cellcolor{lightgray}3.0  & \cellcolor{lightgray}3.25 &   \cellcolor{lightgray}3.5 \\ \hline\hline
		\multicolumn{1}{|c||}{\cellcolor{lightgray}$1$}             & 0.646 & 0.590 & 0.550 & 0.522 & 0.500 & \multicolumn{1}{c||}{0.483} & \multicolumn{1}{c|}{0.718} & 0.695 & 0.674 & 0.656 & 0.641 & 0.629 \\ \hline
		\multicolumn{1}{|c||}{\cellcolor{lightgray}$2$}             & 1.061 & 0.987 & 0.932 & 0.888 & 0.852 & \multicolumn{1}{c||}{0.821} & \multicolumn{1}{c|}{1.241} & 1.278 & 1.291 & 1.290 & 1.283 & 1.272 \\ \hline
		\multicolumn{1}{|c||}{\cellcolor{lightgray}$3$}             & 0.908 & 0.838 & 0.788 & 0.751 & 0.721 & \multicolumn{1}{c||}{0.697} & \multicolumn{1}{c|}{1.053} & 1.074 & 1.082 & 1.082 & 1.079 & 1.073 \\ \hline
		\multicolumn{1}{|c||}{\cellcolor{lightgray}$4$}             & 0.908 & 0.839 & 0.790 & 0.753 & 0.725 & \multicolumn{1}{c||}{0.702} & \multicolumn{1}{c|}{1.055} & 1.081 & 1.094 & 1.099 & 1.099 & 1.097 \\ \hline
		\multicolumn{1}{|c||}{\cellcolor{lightgray}$5$}             & 0.909 & 0.841 & 0.792 & 0.756 & 0.728 & \multicolumn{1}{c||}{0.705} & \multicolumn{1}{c|}{1.057} & 1.087 & 1.104 & 1.113 & 1.116 & 1.117 \\ \hline\hline
		\multicolumn{1}{|c||}{\cellcolor{lightgray} theory}         & 0.900 & 0.833 & 0.786 & 0.750 & 0.722 & \multicolumn{1}{c||}{0.700} & \multicolumn{1}{c|}{1.000} & 1.000 & 1.000 & 1.000 & 1.000 & 1.000 \\ \hline
	\end{tabular}
	\caption{Experimental order of convergence (MINI): $\texttt{EOC}_i(e_{q,i}^{\textup{modular}})$,~${i=1,\dots,5}$.}\vspace*{3mm}
	\label{tab4.2}
	\setlength\tabcolsep{3pt}
	\centering
	\begin{tabular}{c |c|c|c|c|c|c|c|c|c|c|c|c|} \cmidrule(){1-13}
		\multicolumn{1}{|c||}{\cellcolor{lightgray}$\gamma$}	
		& \multicolumn{6}{c||}{\cellcolor{lightgray}Case \hyperlink{case_1}{1}}   & \multicolumn{6}{c|}{\cellcolor{lightgray}Case \hyperlink{case_2}{2}}\\ 
		\hline 
		
		\multicolumn{1}{|c||}{\cellcolor{lightgray}\diagbox[height=1.1\line,width=0.11\dimexpr\linewidth]{\vspace{-0.6mm}$i$}{\\[-5mm] $p$}}
		& \cellcolor{lightgray}2.25 & \cellcolor{lightgray}2.5  & \cellcolor{lightgray}2.75  &  \cellcolor{lightgray}3.0 & \cellcolor{lightgray}3.25  & \multicolumn{1}{c||}{\cellcolor{lightgray}3.5} &  \multicolumn{1}{c|}{\cellcolor{lightgray}2.25}   & \cellcolor{lightgray}2.5  & \cellcolor{lightgray}2.75  & \cellcolor{lightgray}3.0  & \cellcolor{lightgray}3.25 &   \cellcolor{lightgray}3.5 \\ \hline\hline
		\multicolumn{1}{|c||}{\cellcolor{lightgray}$1$}             & 4.102 & 4.480 & 3.881 & 3.963 & 3.486 & \multicolumn{1}{c||}{3.529} & \multicolumn{1}{c|}{5.453} & 4.489 & 5.064 & 4.829 & 5.993 & 5.777 \\ \hline
		\multicolumn{1}{|c||}{\cellcolor{lightgray}$2$}             & 1.029 & 1.026 & 1.022 & 1.017 & 1.012 & \multicolumn{1}{c||}{1.008} & \multicolumn{1}{c|}{1.189} & 1.301 & 1.380 & 1.438 & 1.481 & 1.514 \\ \hline
		\multicolumn{1}{|c||}{\cellcolor{lightgray}$3$}             & 1.004 & 1.001 & 0.998 & 0.995 & 0.992 & \multicolumn{1}{c||}{0.990} & \multicolumn{1}{c|}{1.163} & 1.279 & 1.366 & 1.431 & 1.481 & 1.520 \\ \hline
		\multicolumn{1}{|c||}{\cellcolor{lightgray}$4$}             & 1.008 & 1.007 & 1.006 & 1.004 & 1.003 & \multicolumn{1}{c||}{1.002} & \multicolumn{1}{c|}{1.172} & 1.296 & 1.391 & 1.464 & 1.521 & 1.566 \\ \hline
		\multicolumn{1}{|c||}{\cellcolor{lightgray}$5$}             & 1.009 & 1.009 & 1.009 & 1.008 & 1.008 & \multicolumn{1}{c||}{1.007} & \multicolumn{1}{c|}{1.175} & 1.304 & 1.405 & 1.484 & 1.546 & 1.595 \\ \hline\hline
		\multicolumn{1}{|c||}{\cellcolor{lightgray} theory}         & 0.900 & 0.833 & 0.786 & 0.750 & 0.722 & \multicolumn{1}{c||}{0.700} & \multicolumn{1}{c|}{1.000} & 1.000 & 1.000 & 1.000 & 1.000 & 1.000 \\ \hline
	\end{tabular}
	\caption{Experimental order of convergence (Taylor--Hood): $\texttt{EOC}_i(e_{q,i}^{\textup{norm}})$,~${i=1,\dots,5}$.}\vspace*{3mm}
	\label{tab5.1}
	\setlength\tabcolsep{3pt}
	\centering
	\begin{tabular}{c |c|c|c|c|c|c|c|c|c|c|c|c|} \cmidrule(){1-13}
		\multicolumn{1}{|c||}{\cellcolor{lightgray}$\gamma$}	
		& \multicolumn{6}{c||}{\cellcolor{lightgray}Case \hyperlink{case_1}{1}}   & \multicolumn{6}{c|}{\cellcolor{lightgray}Case \hyperlink{case_2}{2}}\\ 
		\hline 
		
		\multicolumn{1}{|c||}{\cellcolor{lightgray}\diagbox[height=1.1\line,width=0.11\dimexpr\linewidth]{\vspace{-0.6mm}$i$}{\\[-5mm] $p$}}
		& \cellcolor{lightgray}2.25 & \cellcolor{lightgray}2.5  & \cellcolor{lightgray}2.75  &  \cellcolor{lightgray}3.0 & \cellcolor{lightgray}3.25  & \multicolumn{1}{c||}{\cellcolor{lightgray}3.5} &  \multicolumn{1}{c|}{\cellcolor{lightgray}2.25}   & \cellcolor{lightgray}2.5  & \cellcolor{lightgray}2.75  & \cellcolor{lightgray}3.0  & \cellcolor{lightgray}3.25 &   \cellcolor{lightgray}3.5 \\ \hline\hline
		\multicolumn{1}{|c||}{\cellcolor{lightgray}$1$}             & 3.692 & 3.734 & 3.050 & 2.972 & 2.518 & \multicolumn{1}{c||}{2.471} & \multicolumn{1}{c|}{4.908} & 3.741 & 3.979 & 3.622 & 4.328 & 4.044 \\ \hline
		\multicolumn{1}{|c||}{\cellcolor{lightgray}$2$}             & 0.926 & 0.855 & 0.803 & 0.763 & 0.731 & \multicolumn{1}{c||}{0.706} & \multicolumn{1}{c|}{1.070} & 1.084 & 1.085 & 1.079 & 1.070 & 1.060 \\ \hline
		\multicolumn{1}{|c||}{\cellcolor{lightgray}$3$}             & 0.903 & 0.834 & 0.784 & 0.746 & 0.717 & \multicolumn{1}{c||}{0.693} & \multicolumn{1}{c|}{1.047} & 1.066 & 1.073 & 1.073 & 1.070 & 1.064 \\ \hline
		\multicolumn{1}{|c||}{\cellcolor{lightgray}$4$}             & 0.907 & 0.839 & 0.790 & 0.753 & 0.725 & \multicolumn{1}{c||}{0.702} & \multicolumn{1}{c|}{1.055} & 1.080 & 1.093 & 1.098 & 1.099 & 1.097 \\ \hline
		\multicolumn{1}{|c||}{\cellcolor{lightgray}$5$}             & 0.908 & 0.841 & 0.792 & 0.756 & 0.728 & \multicolumn{1}{c||}{0.705} & \multicolumn{1}{c|}{1.057} & 1.086 & 1.104 & 1.113 & 1.116 & 1.117 \\ \hline\hline
		\multicolumn{1}{|c||}{\cellcolor{lightgray} theory}         & 0.900 & 0.833 & 0.786 & 0.750 & 0.722 & \multicolumn{1}{c||}{0.700} & \multicolumn{1}{c|}{1.000} & 1.000 & 1.000 & 1.000 & 1.000 & 1.000 \\ \hline
	\end{tabular}
	\caption{Experimental order of convergence (Taylor--Hood): $\texttt{EOC}_i(e_{q,i}^{\textup{modular}})$,~${i=1,\dots,5}$.}
	\label{tab5.2}
\end{table}


\begin{thebibliography}{33}
	\providecommand{\natexlab}[1]{#1}
	\providecommand{\url}[1]{\texttt{#1}}
	\expandafter\ifx\csname urlstyle\endcsname\relax
	\providecommand{\doi}[1]{doi: #1}\else
	\providecommand{\doi}{doi: \begingroup \urlstyle{rm}\Url}\fi
	
	\bibitem[Amestoy et~al.(2001)Amestoy, Duff, L'Excellent, and Koster]{mumps}
	P.~R. Amestoy, I.~Duff, J.-Y. L'Excellent, and J.~Koster.
	\newblock A fully asynchronous multifrontal solver using distributed dynamic
	scheduling.
	\newblock \emph{SIAM J. Matrix Anal. Appl.}, 23\penalty0 (1):\penalty0 15--41,
	2001.
	\newblock ISSN 0895-4798,1095-7162.
	\newblock \doi{10.1137/S0895479899358194}.
	\newblock URL \url{https://doi.org/10.1137/S0895479899358194}.\enlargethispage{7.5mm}
	
	\bibitem[Arnold et~al.(1984)Arnold, Brezzi, and Fortin]{ABF84}
	D.~N. Arnold, F.~Brezzi, and M.~Fortin.
	\newblock A stable finite element for the {S}tokes equations.
	\newblock \emph{Calcolo}, 21\penalty0 (4):\penalty0 337--344 (1985), 1984.
	\newblock ISSN 0008-0624.
	\newblock \doi{10.1007/BF02576171}.
	\newblock URL \url{https://doi.org/10.1007/BF02576171}.
	
	\bibitem[Bartels(2016)]{Ba16}
	S.~Bartels.
	\newblock \emph{Numerical approximation of partial differential equations},
	volume~64 of \emph{Texts in Applied Mathematics}.
	\newblock Springer, 2016.
	\newblock ISBN 978-3-319-32353-4; 978-3-319-32354-1.
	\newblock \doi{10.1007/978-3-319-32354-1}.
	\newblock URL \url{https://doi.org/10.1007/978-3-319-32354-1}.
	
	\bibitem[Belenki et~al.(2012)Belenki, Berselli, Diening, and R{\r u}{\v z}i{\v
		c}ka]{bdr-phi-stokes}
	L.~Belenki, L.~C. Berselli, L.~Diening, and M.~R{\r u}{\v z}i{\v c}ka.
	\newblock On the {F}inite {E}lement approximation of $p$-{S}tokes systems.
	\newblock \emph{SIAM J. Numer. Anal.}, 50\penalty0 (2):\penalty0 373--397,
	2012.
	
	\bibitem[Bernardi and Raugel(1985)]{BR85}
	C.~Bernardi and G.~Raugel.
	\newblock Analysis of some finite elements for the {S}tokes problem.
	\newblock \emph{Math. Comp.}, 44\penalty0 (169):\penalty0 71--79, 1985.
	\newblock ISSN 0025-5718.
	\newblock \doi{10.2307/2007793}.
	\newblock URL \url{https://doi.org/10.2307/2007793}.
	
	\bibitem[Boffi et~al.(2013)Boffi, Brezzi, and Fortin]{BBF13}
	D.~Boffi, F.~Brezzi, and M.~Fortin.
	\newblock \emph{Mixed finite element methods and applications}, volume~44 of
	\emph{Springer Series in Computational Mathematics}.
	\newblock Springer, Heidelberg, 2013.
	\newblock ISBN 978-3-642-36518-8; 978-3-642-36519-5.
	\newblock \doi{10.1007/978-3-642-36519-5}.
	\newblock URL \url{https://doi.org/10.1007/978-3-642-36519-5}.
	
	\bibitem[Brezzi and Fortin(1991)]{BF1991}
	F.~Brezzi and M.~Fortin.
	\newblock \emph{Mixed and hybrid finite element methods}, volume~15 of
	\emph{Springer Series in Computational Mathematics}.
	\newblock Springer-Verlag, New York, 1991.
	\newblock ISBN 0-387-97582-9.
	\newblock \doi{10.1007/978-1-4612-3172-1}.
	\newblock URL \url{https://doi.org/10.1007/978-1-4612-3172-1}.
	
	\bibitem[Crouzeix and Raviart(1973)]{CR73}
	M.~Crouzeix and P.-A. Raviart.
	\newblock Conforming and nonconforming finite element methods for solving the
	stationary {S}tokes equations. {I}.
	\newblock \emph{Rev. Fran\c{c}aise Automat. Informat. Recherche
		Op\'{e}rationnelle S\'{e}r. Rouge}, 7\penalty0 (no. , no. {\rm
		R}-3):\penalty0 33--75, 1973.
	
	\bibitem[Di~Pietro and Ern(2012)]{DiPE12}
	D.~A. Di~Pietro and A.~Ern.
	\newblock \emph{Mathematical aspects of discontinuous {G}alerkin methods},
	volume~69 of \emph{Math\'{e}matiques \& Applications}.
	\newblock Springer, Heidelberg, 2012.
	\newblock ISBN 978-3-642-22979-4.
	\newblock \doi{10.1007/978-3-642-22980-0}.
	\newblock URL \url{https://doi.org/10.1007/978-3-642-22980-0}.
	
	\bibitem[Diening and Ettwein(2008)]{die-ett}
	L.~Diening and F.~Ettwein.
	\newblock Fractional estimates for non-differentiable elliptic systems with
	general growth.
	\newblock \emph{Forum Math.}, 20\penalty0 (3):\penalty0 523--556, 2008.
	
	\bibitem[Diening and Kreuzer(2008)]{DK08}
	L.~Diening and C.~Kreuzer.
	\newblock Linear convergence of an adaptive finite element method for the
	{$p$}-{L}aplacian equation.
	\newblock \emph{SIAM J. Numer. Anal.}, 46\penalty0 (2):\penalty0 614--638,
	2008.
	\newblock ISSN 0036-1429.
	\newblock \doi{10.1137/070681508}.
	\newblock URL \url{https://doi.org/10.1137/070681508}.
	
	\bibitem[Diening and R{\r u}{\v z}i{\v c}ka(2007)]{dr-interpol}
	L.~Diening and M.~R{\r u}{\v z}i{\v c}ka.
	\newblock Interpolation operators in {O}rlicz--{S}obolev spaces.
	\newblock \emph{Num. Math.}, 107:\penalty0 107--129, 2007.
	\newblock \doi{DOI 10.1007/s00211-007-0079-9}.
	
	\bibitem[Diening et~al.(2013)Diening, Kreuzer, and Schwarzacher]{DKS13}
	L.~Diening, C.~Kreuzer, and S.~Schwarzacher.
	\newblock Convex hull property and maximum principle for finite element
	minimisers of general convex functionals.
	\newblock \emph{Numer. Math.}, 124\penalty0 (4):\penalty0 685--700, 2013.
	\newblock ISSN 0029-599X.
	\newblock \doi{10.1007/s00211-013-0527-7}.
	\newblock URL \url{https://doi.org/10.1007/s00211-013-0527-7}.
	
	\bibitem[Diening et~al.(2014)Diening, Kr\"oner, R{\r u}{\v z}i{\v c}ka, and
	Toulopoulos]{dkrt-ldg}
	L.~Diening, D.~Kr\"oner, M.~R{\r u}{\v z}i{\v c}ka, and I.~Toulopoulos.
	\newblock A {L}ocal {D}iscontinuous {G}alerkin approximation for systems with
	$p$-structure.
	\newblock \emph{IMA J. Num. Anal.}, 34\penalty0 (4):\penalty0 1447--1488, 2014.
	\newblock \doi{doi: 10.1093/imanum/drt040}.
	
	\bibitem[Ern and Guermond(2021)]{EG21}
	A.~Ern and J.~L. Guermond.
	\newblock \emph{Finite Elements I: Approximation and Interpolation}.
	\newblock Number~1 in Texts in Applied Mathematics. Springer International
	Publishing, 2021.
	\newblock ISBN 9783030563417.
	\newblock \doi{10.1007/978-3-030-56341-7}.
	
	\bibitem[Girault and Lions(2001)]{GL01}
	V.~Girault and J.-L. Lions.
	\newblock Two-grid finite-element schemes for the transient {N}avier-{S}tokes
	problem.
	\newblock \emph{M2AN Math. Model. Numer. Anal.}, 35\penalty0 (5):\penalty0
	945--980, 2001.
	\newblock ISSN 0764-583X.
	\newblock \doi{10.1051/m2an:2001145}.
	\newblock URL \url{https://doi.org/10.1051/m2an:2001145}.
	
	\bibitem[Girault and Raviart(1986)]{GR86}
	V.~Girault and P.-A. Raviart.
	\newblock \emph{Finite element methods for {N}avier-{S}tokes equations},
	volume~5 of \emph{Springer Series in Computational Mathematics}.
	\newblock Springer-Verlag, Berlin, 1986.
	\newblock ISBN 3-540-15796-4.
	\newblock \doi{10.1007/978-3-642-61623-5}.
	\newblock URL \url{https://doi.org/10.1007/978-3-642-61623-5}.
	\newblock Theory and algorithms.
	
	\bibitem[Girault and Scott(2003)]{GS03}
	V.~Girault and L.~R. Scott.
	\newblock A quasi-local interpolation operator preserving the discrete
	divergence.
	\newblock \emph{Calcolo}, 40\penalty0 (1):\penalty0 1--19, 2003.
	\newblock ISSN 0008-0624.
	\newblock \doi{10.1007/s100920300000}.
	\newblock URL \url{https://doi.org/10.1007/s100920300000}.
	
	\bibitem[Hunter(2007)]{Hun07}
	J.~D. Hunter.
	\newblock Matplotlib: A 2d graphics environment.
	\newblock \emph{Computing in Science \& Engineering}, 9\penalty0 (3):\penalty0
	90--95, 2007.
	\newblock \doi{10.1109/MCSE.2007.55}.
	
	\bibitem[Jeßberger and Kaltenbach(2023)]{JK23_inhom}
	J.~Jeßberger and A.~Kaltenbach.
	\newblock Finite element discretization of the steady,
        generalized 
	\textcolor{black}{{N}avier-{S}tokes} equations with
        inhomogeneous \textcolor{black}{{D}irichlet} boundary conditions,
	2023.
	
	\bibitem[Kaltenbach and
	R{\r{u}}{\v{z}}i{\v{c}}ka(2023{\natexlab{a}})]{kr-pnse-ldg-2}
	A.~Kaltenbach and M.~R{\r{u}}{\v{z}}i{\v{c}}ka.
	\newblock A {L}ocal {D}iscontinuous {G}alerkin approximation for the
	$p$-{N}avier-{S}tokes system, {P}art~{II}: {C}onvergence rates for the
	velocity.
	\newblock \emph{SIAM J. Num. Anal.}, 61:\penalty0 1641--1663,
	2023{\natexlab{a}}.
	\newblock \doi{10.1137/22M1514751}.
	\newblock URL \url{https://arxiv.org/abs/2208.04107}.
	
	\bibitem[Kaltenbach and
	R{\r{u}}{\v{z}}i{\v{c}}ka(2023{\natexlab{b}})]{kr-pnse-ldg-3}
	A.~Kaltenbach and M.~R{\r{u}}{\v{z}}i{\v{c}}ka.
	\newblock A {L}ocal {D}iscontinuous {G}alerkin approximation for the
	$p$-{N}avier-{S}tokes system, {P}art~{III}: {C}onvergence rates for the
	pressure.
	\newblock \emph{SIAM J. Num. Anal.}, 61:\penalty0 1763--1782,
	2023{\natexlab{b}}.
	\newblock \doi{10.1137/22M1541472}.
	\newblock URL \url{https://arxiv.org/abs/2210.06985}.
	
	\bibitem[Kaltenbach and
	R{\r{u}}{\v{z}}i{\v{c}}ka(2023{\natexlab{b}})]{KR24pressure}
	A.~Kaltenbach and M.~R{\r{u}}{\v{z}}i{\v{c}}ka.
	\newblock Note on quasi-optimal error estimates for the pressure for shear-thickening fluids.
	\newblock \emph{submitted},
	2024{\natexlab{b}}.
	\newblock URL \url{https://arxiv.org/abs/2402.03056}.
	
	\bibitem[Kaplick\'y et~al.(1999)Kaplick\'y, M\'alek, and Star\'a]{KMS2}
	P.~Kaplick\'y, J.~M\'alek, and J.~Star\'a.
	\newblock {$C^{1,\alpha }$}-regularity of weak solutions to a class of
	nonlinear fluids in two dimensions - stationary {D}irichlet problem.
	\newblock \emph{Zap. Nauchn. Sem. Pt. Odel. Mat. Inst.}, 259:\penalty0 89--121,
	1999.
	
	\bibitem[Keast(1986)]{keast}
	P.~Keast.
	\newblock Moderate-degree tetrahedral quadrature formulas.
	\newblock \emph{Comput. Methods Appl. Mech. Engrg.}, 55\penalty0 (3):\penalty0
	339--348, 1986.
	\newblock ISSN 0045-7825,1879-2138.
	\newblock \doi{10.1016/0045-7825(86)90059-9}.
	\newblock URL \url{https://doi.org/10.1016/0045-7825(86)90059-9}.
	
	\bibitem[Kuhn(1960)]{kuhn}
	H.~W. Kuhn.
	\newblock Some combinatorial lemmas in topology.
	\newblock \emph{IBM J. Res. Develop.}, 4:\penalty0 508--524, 1960.
	\newblock ISSN 0018-8646.
	\newblock \doi{10.1147/rd.45.0518}.
	\newblock URL \url{https://doi.org/10.1147/rd.45.0518}.
	
	\bibitem[K\v{r}epela and R\r{u}\v{z}i\v{c}ka(2018)]{kr-nekorn}
	M.~K\v{r}epela and M.~R\r{u}\v{z}i\v{c}ka.
	\newblock A counterexample related to the regularity of the {$p$}-{S}tokes
	problem.
	\newblock \emph{J. Math. Sci. (N.Y.)}, 232\penalty0 (3, Problems in
	mathematical analysis. No. 92 (Russian)):\penalty0 390--401, 2018.
	\newblock \doi{10.1007/s10958-018-3879-9}.
	\newblock URL \url{https://doi.org/10.1007/s10958-018-3879-9}.
	
	\bibitem[K\v{r}epela and R\r{u}\v{z}i\v{c}ka(2020)]{kr-nekorn-add}
	M.~K\v{r}epela and M.~R\r{u}\v{z}i\v{c}ka.
	\newblock {A}ddendum to article: {"}{A} counterexample related to the
	regularity of the {$p$}-{S}tokes problem{"}.
	\newblock \emph{J. Math. Sci. (N.Y.)}, 247\penalty0 (6):\penalty0 957--959,
	2020.
	\newblock \doi{10.1007/s10958-018-3879-9}.
	\newblock URL \url{https://doi.org/10.1007/s10958-018-3879-9}.
	
	\bibitem[Lions(1969)]{lions-quel}
	J.-L. Lions.
	\newblock \emph{Quelques m\'{e}thodes de r\'{e}solution des probl\`emes aux
		limites non lin\'{e}aires}.
	\newblock Dunod, Paris; Gauthier-Villars, Paris, 1969.
	
	\bibitem[Logg and Wells(2010)]{LW10}
	A.~Logg and G.~N. Wells.
	\newblock Dolfin: Automated finite element computing.
	\newblock \emph{ACM Transactions on Mathematical Software}, 37\penalty0
	(2):\penalty0 1--28, 2010.
	\newblock \doi{10.1145/1731022.1731030}.
	
	\bibitem[R{\r u}{\v z}i{\v c}ka and Diening(2007)]{dr-nafsa}
	M.~R{\r u}{\v z}i{\v c}ka and L.~Diening.
	\newblock Non--{N}ewtonian fluids and function spaces.
	\newblock In \emph{Nonlinear Analysis, Function Spaces and Applications,
		Proceedings of {NAFSA} 2006 {P}rague}, volume~8, pages {95--144}, 2007.
	
	\bibitem[Strang and Fix(2008)]{strang_fix}
	G.~Strang and G.~Fix.
	\newblock \emph{An analysis of the finite element method}.
	\newblock Wellesley-Cambridge Press, Wellesley, MA, second edition, 2008.
	\newblock ISBN 978-0-9802327-0-7; 0-9802327-0-8.
	
	\bibitem[Taylor and Hood(1973)]{TH73}
	C.~Taylor and P.~Hood.\enlargethispage{5mm}
	\newblock A numerical solution of the {N}avier-{S}tokes equations using the
	finite element technique.
	\newblock \emph{Internat. J. Comput. \& Fluids}, 1\penalty0 (1):\penalty0
	73--100, 1973.
	\newblock ISSN 0045-7930.
	\newblock \doi{10.1016/0045-7930(73)90027-3}.
	\newblock URL \url{https://doi.org/10.1016/0045-7930(73)90027-3}.\enlargethispage{5mm}
	
	\bibitem[Temam(1977)]{tem}
	R.~Temam.
	\newblock \emph{{Navier-Stokes} Equations}.
	\newblock North-Holland, Amsterdam, 1977.
	
\end{thebibliography}


\end{document}